\documentclass[a4paper,10pt]{amsart}
\usepackage{amsmath,amssymb}
\usepackage{amscd}
\newtheorem{thm}{Theorem}[section]
\newtheorem{prop}[thm]{Proposition}
\newtheorem{lemma}[thm]{Lemma}
\newtheorem{cor}[thm]{Corollary}

\theoremstyle{remark}
\newtheorem{remark}[thm]{Remark}
\newcommand{\id}{{\rm{id}}}

\newcommand{\BN}{\mathbf N}
\newcommand{\BC}{\mathbf C}

\newcommand{\BB}{\mathbf B}
\newcommand{\la}{\langle}
\newcommand{\ra}{\rangle}
\newcommand{\Ind}{{\rm{Ind}}}

\newtheorem{Def}{Definition}[section]

\title{The strong Morita equivalence for inclusions of $C^*$-algebras and conditional expectations
for equivalence bimodules}
\author{Kazunori Kodaka and Tamotsu Teruya}
\address{Department of Mathematical Sciences, Faculty of Science, Ryukyu
\endgraf
University, Nishihara-cho, Okinawa, 903-0213, Japan}
\address{Faculty of Education, Gunma University, 4-2 Aramaki-machi, Maebashi City,
\endgraf
Gunma, 371-8510, Japan}
\address{\sl{E-mail address}: \rm{kodaka@math.u-ryukyu.ac.jp}}
\address{\sl{E-mail address}: \rm{teruya@gunma-u.ac.jp}}

\begin{document}
\maketitle
\begin{abstract}
We shall introduce the notions of the strong Morita equivalence for unital inclusions of unital $C^*$-algebras
and conditional expectations from an equivalence bimodule onto its closed subspace with respect to conditional
expectations from unital $C^*$-algebras onto their unital $C^*$-subalgebras.
Also, we shall study their basic properties.
\end{abstract}

\section{Introduction}\label{sec:intro}In the previous paper \cite {KT3:equivalence}, following Jansen and Waldmann \cite {JW:covariant},
we introduced the notion of the strong Morita equivalence for coactions of a finite dimensional
$C^*$-Hopf algebra on unital $C^*$-algebras. Modifying this notion, we shall introduce the notion of the
strong Morita equivalence for unital inclusions of unital $C^*$-algebras. Also, we shall
introduce the notion of conditional expectations from an equivalence bimodule onto its closed subspace
with respect to conditional expectations from unital $C^*$-algebras onto their unital $C^*$-subalgebras.
Furthermore, we shall study their basic properties.
\par
To specify, let $A$ and $B$ be
unital $C^*$-algebras and $H$ a finite dimensional
$C^*$-Hopf algebra. Let $H^0$ be its dual $C^*$-Hopf algebra. Let $\rho$ and $\sigma$ be
coactions of $H^0$ on $A$ and $B$, respectively.
Then we can obtain the unital inclusions $A\subset A\rtimes_{\rho}H$ and $B\subset B\rtimes_{\sigma}H$
and the canonical conditional expectations $E_1^{\rho}$ and $E_1^{\sigma}$ from $A\rtimes_{\rho}H$
and $B\rtimes_{\sigma}H$ onto $A$ and $B$, respectively.
We suppose that $\rho$ and $\sigma$ are strongly Morita equivalent.
Then there are
an $A-B$-equivalence bimodule $X$ and a coaction $\lambda$ of $H^0$
on $X$ with respect to $(A, B, \rho, \sigma)$. Let $E^{\lambda}$ be the linear map
from $X\rtimes_{\lambda}H$ onto $X$ defined by
$$
E_1^{\lambda}(x\rtimes_{\lambda}h)=\tau(h)x
$$
for any $x\in X$, $h\in H$, where $\tau$ is the Haar trace on $H$.
\par
In Section \ref{sec:definition}, we give the notion of the strong Morita equivalence for unital
inclusions of unital $C^*$-algebras so that $A\subset A\rtimes_{\rho}H$
and $B\subset B\rtimes_{\sigma}H$ are strongly Morita equivalent.
We also give the notion of conditional expectations from an equivalence bimodule onto
its closed subspace with respect to conditional expectations from unital $C^*$-algebras onto their
unital $C^*$-subalgebras so that $E^{\lambda}$
is a conditional expectation from $X\rtimes_{\lambda}H$ onto $X$ with respect to $E^A$ and $E^B$.
\par
In Sections \ref {sec:one-sided}, \ref {sec:example} and \ref{sec:linking}, we study
the properties of conditional expectations from an equivalence bimodule onto
its closed subspace with respect to conditional expectations from unital $C^*$-algebras
onto their unital $C^*$-subalgebras. In Sections \ref{sec:upward}, \ref{sec:duality} and \ref{sec:downward}, we give the
upward and downward basic constructions for a conditional expectation from
an equivalence bimodule onto its closed subspace and a duality result which are similar to the ordinary
basic constructions for conditional expectations from unital $C^*$-algebras onto
their unital $C^*$-subalgebras. Furthermore, in Section \ref{sec:between}, we study a relationship
between the upward basic construction and the downward basic construction for the conditional expectation
from an equivalence bimodule onto its closed subspace. Finally
In Section \ref{sec:paragroup}, we show that the strong Morita equivalence for unital inclusions
of unital $C^*$-algebras preserves their paragroups.
\par
Let $A$ and $B$ be $C^*$-algebras and $X$ an $A-B$-bimodule.
Then we denote its left $A$-action and right $B$-action
on $X$ by $a\cdot x$ and $x\cdot b$ for any $a\in A$, $b\in B$ and $x\in X$.
For a $C^*$-algebra $A$, we denote by $M_n (A)$ the $n\times n$-matrix algebra over $A$ and
$I_n$ denotes the unit element in $M_n (\BC)$. We identify $M_n (A)$ with $A\otimes M_n (\BC)$.

\section{The strong Morita equivalence and basic properties}\label{sec:definition}
We begin this section with the following definition: Let $A, B, C$ and $D$ be $C^*$-algebras.

\begin{Def}\label{def:inclusion}Inclusions of $C^*$-algebras $A\subset C$ and $B\subset D$
with $\overline{AC}=C$ and $\overline{BD}=D$ are
\it
strongly Morita equivalent
\rm
if there are a $C-D$-equivalence bimodule $Y$ and its closed subspace $X$ satisfying the
following conditions:
\newline
$(1)$ $a\cdot x\in X$, ${}_C \la x, y \ra\in A$ for any $a\in A$, $x, y\in X$ and
$\overline{{}_C \la X, X \ra}=A$, $\overline{{}_C \la Y, X \ra}=C$,
\newline
$(2)$ $x\cdot b\in X$, $\la x, y \ra_B \in B$ for any $b\in B$, $x, y\in X$ and
$\overline{\la X, X \ra_D }=B$, $\overline{ \la Y, X \ra_D }=D$.
\newline
Then we say that the inclusion $A\subset C$ are strongly Morita equivalent to the inclusion $B\subset D$ with respect to
the $C-D$-equivalent bimodule $Y$ and its closed subspace $X$. We note that $X$ can be regarded as an
$A-B$-equivalence bimodule. 
\end{Def}

\begin{remark}\label{remark:closure} (1) If $Y$ is a $C-D$-equivalence bimodule,
$\overline{C\cdot Y}=\overline{Y\cdot D}=Y$ by Brown, Mingo and Shen \cite [Proposition 1.7]{BMS:quasi}.
\newline
(2) If strongly Morita equivalent inclusions $A\subset C$ and $B\subset D$ are unital
inclusions of unital $C^*$-algebras, we do not need to take the closure in Definition \ref{def:inclusion}.
\end{remark}
\begin{prop}\label{prop:relation}The strong Morita equivalence for inclusions of
$C^*$-algebras is equivalence relation.
\end{prop}
\begin{proof}It suffices to show the transitivity since the other conditions clearly
hold. Let $A\subset C$ and $B\subset D$ and $K\subset L$ be inclusions of
$C^*$-algebras. We suppose that $A\subset C$ is strongly
Morita equivalent to $B\subset D$ with respect to a $C-D$-equivalence bimodule
$Y$ and its closed subspace $X$ and that $B\subset D$ is strongly Morita
equivalent to $K\subset L$ with respect to a $D-L$-equivalence bimodule
$W$ and its closed subspace $Z$. We consider the closed subspace of $Y\otimes_D W$
spanned by the set
$$
\{ x\otimes z\in Y\otimes_D W \, | \, x\in X, \quad z\in Z \} .
$$
We denote it by $X\otimes_D Z$. For any $x_1 , x_2 \in X$, $z_1 , z_2 \in Z$ and
$a\in A$, $k\in K$,
\begin{align*}
a\cdot (x_1 \otimes z_1 ) &=(a\cdot x_1 )\otimes z_1 \in X\otimes_D Z , \\
(x_1 \otimes z_1 )\cdot k & =x_1 \otimes (z_1 \cdot k) \in X\otimes_D Z , \\
{}_C \la x_1 \otimes z_1 \, ,x_2 \otimes z_2 \ra & ={}_C \la x_1 \cdot {}_D \la z_1 , \, z_2 \ra ,\, x_2 \ra
={}_C \la x_1 \cdot {}_B \la z_1 , \, z_2 \ra ,\, x_2 \ra \\
&  ={}_A \la x_1 \cdot {}_B \la z_1 , \, z_2 \ra ,\, x_2 \ra \in A , \\
\la x_1 \otimes z_1 \, , x_2 \otimes z_2 \ra_L & =\la z_1, \, \la x_1, \, x_2 \ra_D \cdot z_2 \ra_L 
=\la z_1, \, \la x_1, \, x_2 \ra_B \cdot z_2 \ra_L \\
& =\la z_1, \, \la x_1, \, x_2 \ra_B \cdot z_2 \ra_K \in K .
\end{align*}
Also, by Definition \ref{def:inclusion} and Remark \ref{remark:closure}, 
\begin{align*}
\overline{{}_C \la X\otimes_D Z, \, X\otimes_D Z \ra} & =\overline{{}_C \la X\cdot {}_B  \la Z, \, Z \ra , \, X \ra}
=\overline{{}_A \la X\cdot B , \, X \ra} =\overline{{}_A \la X , \, X \ra} =A , \\
\overline{\la X\otimes_D Z , \, X\otimes_D Z \ra_L} & =\overline{\la Z, \, \la X, \, X \ra_B\cdot Z \ra_L}
=\overline{\la Z, \, B\cdot Z \ra_K} =\overline{\la Z, \, Z \ra_K }=K , \\
\overline{{}_C \la Y\otimes_D W, \, X\otimes_D Z \ra} & =\overline{{}_C \la Y\cdot {}_D  \la W, \, Z \ra , \, X \ra}
=\overline{{}_C \la Y\cdot D , \, X \ra} =\overline{{}_C \la Y , \, X \ra} =C , \\
\overline{\la Y\otimes_D W , \, X\otimes_D Z \ra_L } & =\overline{\la W, \, \la Y, \, X \ra_D \cdot Z \ra_L}
=\overline{\la W, \, D\cdot Z \ra_L} =\overline{\la D\cdot W, \,  Z \ra_L} \\
& =\overline{\la W, \, Z \ra_L} =L .
\end{align*}
Therefore, we obtain the conclusion.
\end{proof}
Let $A\subset C$ and $B\subset D$ be unital inclusions of unital $C^*$-algebras which are
strongly Morita equivalent with respect to a $C-D$-equivalence bimodule $Y$ and its closed subspace $X$.
Let $E^A$ and $E^B$ be conditional expectations from $C$ and $D$ onto $A$ and $B$,
respectively. Let $E^X$ be a linear map
from $Y$ onto $X$.

\begin{Def}\label{def:expectation}With above notations,
we say that $E^X$ is a
\sl
conditional expectation
\rm
from $Y$ onto $X$ with respect to $E^A$ and $E^B$
if $E^X$ satisfies the following conditions:
\newline
$(1)$ $E^X (c\cdot x)=E^A (c)\cdot x$ for any $c\in C$, $x\in X$,
\newline
$(2)$ $E^X (a\cdot y)=a\cdot E^X (y)$ for any $a\in A$, $y\in Y$,
\newline
$(3)$ $E^A ({}_C \la y, x \ra)={}_C \la E^X (y), x \ra$ for any $x\in X$, $y\in Y$,
\newline
$(4)$ $E^X (x\cdot d)=x\cdot E^B (d)$ for any $d\in D$ $x\in X$,
\newline
$(5)$ $E^X (y\cdot b)=E^X (y)\cdot b$ for any $b\in B$, $y\in Y$,
\newline
$(6)$ $E^B (\la y, x \ra_D )=\la E^X (y), x \ra_D$ for any $x\in X$, $y\in Y$.
\end{Def}
By Definition \ref{def:inclusion}, we can see that
$E^A ({}_C \la y, x \ra)={}_A \la E^X (y), x \ra$ for any $x\in X$, $y\in Y$ and that
$E^B (\la y, x \ra_D )=\la E^X (y), x \ra_B$ for any $x\in X$, $y\in Y$.

Let $A\subset C$ and $B\subset D$ be unital inclusions of unital $C^*$-algebras, which
are strongly Morita equivalent with respect to a $C-D$-equivalence bimodule $Y$ and
its closed subspace $X$. By Kajiwara and Watatani \cite [Lemma 1.7 and Corollary 1.28]{KW1:bimodule},
there are elements $x_1, \dots, x_n\in X$ such that $\sum_{i=1}^n \la x_i , x_i \ra_B =1$.
We consider $X^n$ as an $M_n (A)-B$-equivalence bimodule in the evident way and
let $\overline{x}=(x_1 , x_2 , \dots , x_n )\in X^n$. Then $\la \overline{x}, \, \overline{x} \ra_B =1$.
Let $p={}_{M_n (A)} \la \overline{x} , \, \overline{x} \ra$ and $z={}_{M_n (A)} \la \overline{x} , \, \overline{x} \ra 
\cdot\overline{x}$.
Also, let $\Psi_B$ be the map from $B$ to $M_n (A)$
defined by
$$
\Psi_B (b)={}_{M_n (A)} \la z\cdot b, \, z \ra=[{}_A \la x_i b , \, x_j \ra]_{ij=1}^n
$$
for any $b\in B$.
Then $p$ is a full projection in $M_n (A)$, that is, $M_n (A)pM_n (A)=M_n (A)$
and $\Psi_B $ is an isomorphism of $B$ onto $pM_n (A)p$ by the proof of
Rieffel \cite [Proposition 2.1]{Rieffel:rotation}. We repeat the above discussions
for the $C-D$-equivalence bimodule $Y$ in the following way:
We note that
$$
\sum_{i=1}^n \la x_i \, , x_i \ra_D =\sum_{i=1}^n \la x_i , \, x_i \ra_B =1 .
$$
We consider $Y^n$ as an $M_n (C)-D$-equivalence bimodule in the evident way.
Then $\overline{x}=(x_1 , \dots, x_n )\in Y^n$ and
\begin{align*}
p & ={}_{M_n (A)} \la \overline{x}, \, \overline{x} \ra ={}_{M_n (C)} \la \overline{x}, \, \overline{x} \ra \in M_n (C) , \\
z & ={}_{M_n (A)} \la \overline{x}, \, \overline{x} \ra \cdot \overline{x}
={}_{M_n (C)} \la \overline{x}, \, \overline{x} \ra \cdot \overline{x}\in Y^n .
\end{align*}
Let $\Psi_D$ be the map from $D$ to $M_n (C)$ defined by
$$
\Psi_D (d)={}_{M_n (C)} \la z\cdot d, \, z \ra
$$
for any $d\in D$. By the proof of \cite [Proposition 2.1]{Rieffel:rotation} $p$ is a full projection
in $M_n (C)$, that is, $M_n (C)pM_n (C)=M_n (C)$ and $\Psi_D$ is an isomorphism of $D$ onto $pM_n (C)p$.
Also, we see that $\Psi_B =\Psi_D |_B$ by the definitions of $\Psi_B$ and $\Psi_D$.
Let $\Psi_X$ be the map from $X$ to $M_n (A)$ defined by
$$
\Psi_X (x) = \begin{bmatrix} {}_A \la x, \, x_1 \ra & {}_A \la x, \, x_2 \ra & \ldots & {}_A \la x, \, x_n \ra \\
0 & 0 & \ldots & 0 \\
\vdots & \vdots & \ddots & \vdots \\
0 & 0 & \ldots & 0 \end{bmatrix}_{n\times n}
$$
for any $x\in X$. Let $f=\begin{bmatrix} 1 & 0 & \ldots & 0 \\
0 & 0 & \ldots & 0 \\
\vdots & \vdots & \ddots & \vdots \\
0 & 0 & \ldots & 0 \end{bmatrix}_{n\times n}$.

\begin{lemma}\label{lem:bijection}With the above notations, $\Psi_X$ is a bijective linear map
from $X$ onto $(1\otimes f)M_n (A)p$.
\end{lemma}
\begin{proof}It is clear that $\Psi_X$ is linear and that $(1\otimes f)\Psi_X (x)=\Psi_X (x)$ for any $x\in X$.
We note that $p=[{}_A \la x_i , \, x_j \ra ]_{i,j=1}^n$. Then for any $x\in X$
\begin{align*}
\Psi_X (x)p & =\begin{bmatrix} \sum_{i=1}^n {}_A \la x, \, x_i \ra \,{}_A \la x_i , \, x_1 \ra  & \ldots
& \sum_{i=1}^n {}_A \la x, \, x_i \ra \, {}_A \la x_i , \, x_n \ra \\
0 & \ldots & 0 \\
\vdots & \ddots & \vdots \\
0 & \ldots & 0 \end{bmatrix}_{n\times n} .
\end{align*}
Here for $j=1,2,\dots, n$
$$
\sum_{i=1}^n {}_A \la x, \, x_i \ra \,{}_A \la x_i , \, x_j \ra =\sum_{i=1}^n {}_A \la {}_A \la x \, , x_i  \ra \cdot x_i , \, x_j \ra
=\sum_{i=1}^n {}_A \la x \cdot \la x_i , \, x_i \ra_B , \, x_j \ra 
={}_A \la x, \, x_j \ra .
$$
Thus we can see that $\Psi_X (x)p=\Psi_X (x)$ for any $x\in X$. Hence $\Psi_X$ is the linear map from
$X$ to $(1\otimes f)M_n (A)p$. Let $y\in (1\otimes f)M_n (A)p$. Then we can write that
$$
y=\begin{bmatrix} y_1 & \ldots & y_n \\
0 & \ldots & 0 \\
\vdots & \ddots & \vdots \\
0 & \ldots & 0 \end{bmatrix}p
=\begin{bmatrix} \sum_{i=1}^n y_i \, {}_A \la x_i , \, x_1 \ra & \ldots & \sum_{i=1}^n y_i  \, {}_A \la x_i , \, x_n \ra \\
0 & \ldots & 0 \\
\vdots & \ddots & \vdots \\
0 & \ldots & 0 \end{bmatrix} ,
$$
where $y_1, \dots , y_n \in A$. Modifying Remark after \cite [Lemma 1.11]{KW1:bimodule},
let $\chi$ be the linear map from $(1\otimes f)M_n (A)p$ to $X$ defined by
$$
\chi(y)=\sum_{ij=1}^n y_i \, {}_A \la x_i, \, x_j \ra \cdot x_j .
$$
Then since $\sum_{j=1}^n \la x_j ,\, x_j \ra_B =1$,
\begin{align*}
& (\Psi_X \circ\chi )(y) \\
& =\begin{bmatrix} {}_A \la \sum_{ij=1}^n y_i \, {}_A  \la x_i , \, x_j \ra \cdot x_j , \, x_1 \ra  & \ldots &
{}_A \la \sum_{ij=1}^n y_i \, {}_A \la x_i , \, x_j \ra \cdot x_j , \,x_n \ra  \\
0 & \ldots & 0 \\
\vdots & \ddots & \vdots \\
0 & \ldots & 0 \end{bmatrix} \\
& =\begin{bmatrix} {}_A \la \sum_{ij=1}^n y_i \cdot x_i  \cdot \la  x_j , \,  x_j \ra_B , \, x_1 \ra  & \ldots &
{}_A \la \sum_{ij=1}^n y_i \cdot x_i \cdot \la x_j , \,  x_j \ra_B , \, x_n \ra  \\
0 & \ldots & 0 \\
\vdots & \ddots & \vdots \\
0 & \ldots & 0 \end{bmatrix} \\
& =y.
\end{align*}
Also,
\begin{align*}
(\chi\circ\Psi_X )(x) & =\sum_{ij=1}^n {}_A \la x,\, x_i \ra \, {}_A \la x_i , \, x_j \ra \cdot x_j
=\sum_{ij=1}^n {}_A \la x, \, x_i \ra \cdot x_i \cdot \la x_j , \, x_j \ra_B \\
& =\sum_{i=1}^n {}_A \la x, \, x_i \ra \cdot x_i =\sum_{i=1}^n x \cdot \la x_i , \, x_i \ra_B =x .
\end{align*}
Thus we obtain the conclusion.
\end{proof}

\begin{lemma}\label{lem:condition}With the above notations, $\Psi_X$ satisfies the following:
\newline
$(1)$ $\Psi_X (a\cdot x)=a\cdot\Psi_X (x)$ for any $a\in A$, $x\in X$,
\newline
$(2)$ $\Psi_X (x\cdot b)=\Psi_X (x)\cdot \Psi_B (b)$ for any $b\in B, x\in X$,
\newline
$(3)$ ${}_A \la \Psi_X (x), \Psi_X (y) \ra ={}_A \la x, \, y \ra$ for any $x,y\in X$,
\newline
where we identify $A$ with $(1\otimes f)M_n (A)(1\otimes f)=A\otimes f$,
\newline
$(4)$ $\la \Psi_X (x), \Psi_X (y) \ra_{pM_n (A)p}=\Psi_B (\la x, \, y \ra_B )$ for any $x, y\in X$.
\end{lemma}
\begin{proof}(1) Let $a\in A$ and $x\in X$. Then
$$
\Psi_X (a\cdot x)=\begin{bmatrix} {}_A \la a\cdot x, \, x_1 \ra & \ldots & {}_A \la a\cdot x, \, x_n \ra \\
0 & \ldots & 0 \\
\vdots & \ddots & \vdots \\
0 & \ldots & 0 \end{bmatrix}=a\cdot \Psi_X (x) .
$$
Hence we obtain (1).
\newline
(2) Let $b\in B$ and  $x\in X$. Then
\begin{align*}
& \Psi_X (x)\cdot \Psi_B (b) =\begin{bmatrix} {}_A \la x, \, x_1 \ra & \ldots & {}_A \la x, \, x_n \ra \\
0 & \ldots & 0 \\
\vdots & \ddots & \vdots \\
0 & \ldots & 0 \end{bmatrix}_{n\times n} [{}_A \la x_i \cdot b, \, x_j \ra]_{ij=1}^n \\
& =\begin{bmatrix} \sum_{i=1}^n {}_A \la x, \, x_i \ra \, {}_A \la x_i \cdot b, \, x_1 \ra & \ldots & \sum_{i=1}^n {}_A 
\la x, \, x_i \ra \, {}_A \la x_i \cdot b, \, x_n \ra \\
0 & \ldots & 0 \\
\vdots & \ddots & \vdots \\
0 & \ldots & 0 \end{bmatrix}_{n\times n} .
\end{align*}
Here for $j=1,2,\dots,n$,
$$
\sum_{i=1}^n {}_A \la x, \, x_i \ra \, {}_A \la x_i \cdot b, \, x_j \ra
=\sum_{i=1}^n {}_A \la x \cdot \la x_i, \, x_i \ra_B b , \, x_j \ra
={}_A \la x\cdot b, x_j \ra .
$$
Thus we obtain (2).
\newline
(3) Let $x, y\in X$. Then since we identify $A$ with $A\otimes f$,
\begin{align*}
{}_A \la \Psi_X (x), \Psi_X (y) \ra & = \sum_{i=1}^n{}_A \la x, \, x_i \ra\, {}_A \la y, \, x_i \ra^*
=\sum_{i=1}^n {}_A \la x, \, x_i \ra \, {}_A \la x_i , \, y \ra \\
& =\sum_{i=1}^n {}_A \la \, {}_A \la x, \, x_i \ra \cdot x_i , \, y \ra 
=\sum_{i=1}^n {}_A  \la x \cdot \la x_i , \, x_i \ra_B , \, y \ra
={}_A \la x, \, y \ra .
\end{align*}
Hence we obtain (3).
\newline
(4) Let $x, y\in X$. Then
$$
\la \Psi_X (x), \, \Psi_X (y) \ra_{pM_n (A)p} =\Psi_X (x)^* \Psi_X (y)
=[{}_A \la x, \, x_i \ra^* {}_A \la y\, , x_j \ra ]_{ij=1}^n .
$$
On the other hand,
\begin{align*}
\Psi_B (\la x, \, y \ra_B ) & =[{}_A \la x_i \cdot \la x, \, y \ra_B , \, x_j \ra ]_{ij=1}^n
=[{}_A \la {}_A \la x_i , \, x \ra \cdot y, \, x_j \ra ]_{ij}^n \\
& =[{}_A \la x_i , \, x \ra {}_A \la y, \, x_j \ra ]_{ij=1}^n .
\end{align*}
Hence we obtain (4).
\end{proof}

Let $\Psi_Y$ be the map from $Y$ to $M_n (C)$ defined by
$$
\Psi_Y (x)=\begin{bmatrix} {}_C \la x, \, x_1 \ra & \ldots & {}_C \la x, \, x_n \ra \\
0 & \ldots & 0 \\
\vdots & \ddots & \vdots \\
0 & \ldots & 0 \end{bmatrix}_{n\times n}
$$
for any $x\in Y$. 

\begin{cor}\label{cor:restriction}With the above notations, $\Psi_Y$ is a bijective linear map
from $Y$ onto $(1\otimes f)M_n (C)p$ satisfying the following:
\newline
$(1)$ $\Psi_Y (c\cdot x)=c\cdot\Psi_Y (x)$ for any $c\in C$, $x\in Y$,
\newline
$(2)$ $\Psi_Y (x\cdot d)=\Psi_Y (x)\cdot \Psi_D (d)$ for any $d\in D$, $x\in Y$,
\newline
$(3)$ ${}_C \la \Psi_Y (x), \Psi_Y (y) \ra ={}_C \la x, y \ra$ for any $x, y\in Y$,
\newline
where we identify $C$ with $(1\otimes f)M_n (C)(1\otimes f)=C\otimes f$,
\newline
$(4)$ $\la \Psi_Y (x), \Psi_Y (y) \ra_{pM_n (C)p} =\Psi_D (\la x, y \ra_D )$ for any $x, y\in Y$,
\newline
$(5)$ $\Psi_X =\Psi_Y |_X$.
\end{cor}
\begin{proof}It is clear that $\Psi_X =\Psi_Y |_X$ by the definitions of $\Psi_X$ and $\Psi_Y$.
By Lemmas \ref{lem:bijection} and \ref{lem:condition}, we obtain the others.
\end{proof}

Let $A\subset C$ and $B\subset D$ be unital inclusions of unital $C^*$-algebras.
We suppose that $A\subset C$ and $B\subset D$ are strongly Morita equivalent with respect to
a $C-D$-equivalence bimodule $Y$ and its closed subspace $X$. Then by Lemmas \ref{lem:bijection},
\ref{lem:condition} and Corollary \ref{cor:restriction}, we may assume that
$$
B=pM_n (A)p, \quad D=pM_n (C)p, \quad Y=(1\otimes f)M_n (C)p, \quad X=(1\otimes f)M_n (A)p,
$$
where $p$ is a projection in $M_n (A)$ satisfying that $M_n (A)pM_n (A)=M_n (A)$, that is, $p$
is a full in $M_n (A)$
and $n$ is a positive integer. We regard $X$ and $Y$ as an $A-pM_n (A)p$-equivalence bimodule and
a $C-pM_n (C)p$-equivalence bimodule in the usual way.
\par
We consider the following: Let $A\subset C$ be a unital inclusion of unital
$C^*$-algebras and $p$ a  full projection in $M_n (A)$.
Then the inclusion $pM_n (A)p\subset pM_n (C)p$
is strongly Morita equivalent to $A\subset C$ with respect to the $C-pM_n (C)p$-equivalence
bimodule $(1\otimes f)M_n (C)p$ and its closed subspace $(1\otimes f)M_n (A)p$.
Let $E^A$ be a conditional expectation of Watatani index-finite type from $C$ onto $A$.
We denote by $\Ind_W (E^A )$ the Watatani index of $E^A$.
We note that $\Ind_W (E^A )\in C\cap C'$. Let $\{(u_i, u_i^* )\}_{i=1}^N$ be a quasi-basis
for $E^A$. Then $\{(u_i \otimes I_n , \, u_i^* \otimes I_n )\}_{i=1}^N$ is a quasi-basis
for $E^A \otimes\id$, the conditional expextation from $M_n (C)$ onto $M_n (A)$.
Since $p$ is a full projection in $M_n (A)$, there is elements $a_1, \dots, a_K , \, b_1, \cdots, b_K$
in $M_n (A)$ such that $\sum_{i=1}^K a_i pb_i =1_{M_n (A)}$.
Let $E_p^A$ be the conditional expectation from $pM_n (C)p$ onto $pM_n(A)p$
defined by
$$
E_p^A (x)=(E^A \otimes\id)(x)
$$
for any $x\in pM_n (A)p$. Then by routine computations, we can see that
$$
\{(p(u_i \otimes I_n )a_j p , \, pb_j (u_i^* \otimes I_n )p )\}_{i=1,\dots, N, \, j=1,\dots, K}
$$
is a quasi-basis for $E_p^A$. Furthermore,
\begin{align*}
\Ind_W (E_p^A ) & =\sum_{i,j}p(u_i \otimes I_n )a_j pb_j (u_i^* \otimes I_n )p
=\sum_{i}p(u_i u_i^* \otimes I_n )p \\
& =p(\Ind_W (E^A )\otimes I_n )p=(\Ind_W (E^A )\otimes I_n )p .
\end{align*}
Let $F$ be the linear map from $(1\otimes f)M_n (C)p$ onto $(1\otimes f)M_n (A)p$
defined by
$$
F((1\otimes f)xp)=(E^A \otimes\id)((1\otimes f)xp)=(1\otimes f)(E^A \otimes\id)(x)p
$$
for any $x\in M_n (C)$.

\begin{lemma}\label{lem:expectation}With the above notations, $F$ is a conditional
expectation from $(1\otimes f)M_n (C)p$ onto $(1\otimes f)M_n (A)p$ with respect to
$E^A$ and $E_p^A$.
\end{lemma}
\begin{proof}
It suffices to show that $F$ satisfies Conditions (1)-(6) in Definition \ref{def:expectation}.
\newline
(1) For any $c\in C$, $x\in M_n (A)$,
\begin{align*}
F(c\cdot (1\otimes f)xp)& =F((c\otimes f)xp)=F((1\otimes f)(c\otimes I_n )xp) \\
& =(1\otimes f)(E^A \otimes \id)((c\otimes I_n )x)p=(1\otimes f)(E^A (c)\otimes I_n )xp \\
& =E^A (c)\cdot (1\otimes f)xp .
\end{align*}
Thus we obtain Condition (1) in Definition \ref{def:expectation}.
\newline
(2) For any $a\in A$, $y\in M_n (C)$,
\begin{align*}
F(a\cdot (1\otimes f)yp) & =F((1\otimes f)(a\otimes I_n )yp)=(1\otimes f)(E^A \otimes \id)((a\otimes I_n )y)p \\
& =a\cdot (1\otimes f)(E^A \otimes\id)(y)p=a\cdot F((1\otimes f)yp) .
\end{align*}
Thus we obtain Condition (2) in Definition \ref{def:expectation}.
\newline
(3) For any $x\in M_n (A), y\in M_n (C)$,
\begin{align*}
{}_C \la F((1\otimes f)yp), \, (1\otimes f)xp \ra & ={}_C \la (1\otimes f)(E^A \otimes\id)(y)p, \, (1\otimes f)xp \ra \\
& =(1\otimes f)(E^A \otimes\id)(y)px^* (1\otimes f) \\
& =(E^A \otimes\id)((1\otimes f)ypx^* (1\otimes f)) \\
& =(E^A \otimes\id)({}_C \la (1\otimes f)yp, \, (1\otimes f)xp \ra )
\end{align*}
since we identify $C$ with $(1\otimes f)M_n (C)(1\otimes f)=C\otimes f$.
Thus we obtain Condition (3) in Definition \ref{def:expectation}.
\newline
(4) For any $y\in M_n (C)$, $x\in M_n (A)$,
\begin{align*}
F((1\otimes f)xp\cdot pyp) & =F((1\otimes f)xpyp)=(1\otimes f)(E^A \otimes\id)(xpy)p \\
& =(1\otimes f)xp(E^A \otimes\id)(y)p=(1\otimes f)xp\cdot E_p^A (pyp) .
\end{align*}
Thus we obtain Condition (4) in Definition \ref{def:expectation}.
\newline
(5) For any $x\in M_n (A)$, $y\in M_n (C)$,
\begin{align*}
F((1\otimes f)yp\cdot pxp) & =F((1\otimes f)ypxp)=(1\otimes f)(E^A \otimes \id)(ypx)p \\
& =(1\otimes f)(E^A \otimes\id)(y)p\cdot pxp =F((1\otimes f)yp)\cdot pxp.
\end{align*}
Thus we obtain Condition (5) in Definition \ref{def:expectation}.
\newline
(6) For any $x\in M_n (A), y\in M_n (C)$,
\begin{align*}
\la F((1\otimes f)yp), \, (1\otimes f)xp \ra_{pM_n (C)p} & =p(E^A \otimes\id)(y)^* (1\otimes f)xp \\
& =p(E^A \otimes\id)(y^* (1\otimes f)x)p \\
& =E_p^A ( \la (1\otimes f)yp, \, (1\otimes f)xp \ra_{pM_n (C)p} ) .
\end{align*}
Thus we obtain Condition (6) in Definition \ref{def:expectation}. Therefore, we obtain the conclusion.
\end{proof}

\begin{thm}\label{thm:morita}Let $A\subset C$ and $B\subset D$ be unital inclusions of
unital $C^*$-algebras, which are strongly Morita equivalent with respect to
a $C-D$-equivalence bimodule $Y$ and its closed subspace $X$. If there is a conditional
expectation $E^A$ of Watatani index-finite type from $C$ onto $A$, then
there are a conditional
expectation $E^B$ of Watatani index-finite type from $D$ onto $B$
and a conditional expectation $E^X$ from $Y$ onto $X$ with respect to $E^A$ and $E^B$.
Also, if there is a conditional
expectation $E^B$ of Watatani index-finite type from $D$ onto $B$, then we have the same result as
above.
\end{thm}
\begin{proof}This is immediate by Lemmas \ref{lem:bijection}, \ref{lem:condition}, \ref{lem:expectation}
and Corollary \ref{cor:restriction}.
\end{proof}

\section{One-sided conditional expectations on full Hilbert $C^*$-modules}
\label{sec:one-sided}Let $B\subset D$ be a unital inclusion of unital $C^*$-algebras and
let $Y$ be a full right Hilbert $D$-module and $X$ its closed subspace satisfying the following:
\newline
(1) $x\cdot b\in X$, $\la x, y \ra_D\in B$ for any $b\in B$, $x, y\in X$,
\newline
(2) $\overline{\la X, X \ra_D }=B$, $\overline{\la Y, X \ra_D}=D$,
\newline
(3) There is a finite set $\{x_i \}_{i=1}^n \subset X$ such that for any $y\in Y$
$$
\sum_{i=1}^n x_i \cdot \la x_i , \, y \ra_D =y .
$$
We note that $Y$ is of finite type and that $X$ can be regarded as a full right Hilbert $B$-module of finite type
in the sense of Kajiwara and Watatani \cite {KW1:bimodule}.
Let $\BB_D (Y)$ be the $C^*$-algebra of all right $D$-linear operators on $Y$ for
which has a right adjoint $D$-linear operator on $Y$.
Let $C=\BB_D (Y)$.
For any $x, y\in Y$,
let $\theta_{x, y}^Y$ be the rank-one operator on $Y$ defined by
$$
\theta_{x, y}^Y (z)=x \cdot \la y, z \ra_D
$$
for any $z\in Y$. Then $\theta_{x, y}^Y$ is a right $D$-module operator.
Hence $\theta_{x, y}^Y \in C$ for any $x, y\in Y$. Since $D$ is unital, by \cite [Lemma 1.7]{KW1:bimodule},
$C$ is the $C^*$-algebra of all linear spans of such $\theta_{x, y}^Y$. Let $A_0$ be the linear spans of 
the set $\{\theta_{x, y}^Y \, | \, x, y\in X \}$. By the assumptions, $\sum_{i=1}^n \theta_{x_i, x_i }^Y =1_Y$.
Hence $A_0$ is a $*$-algebra. Let $A$ be the closure of $A_0$ in $\BB_D (Y)$. Then $A$ is a unital
$C^*$-subalgebra of $C$. Let $\BB_B (X)$ be the $C^*$-algebra defined in the same way as above.
Let $\pi$ be the map from $\BB_B (X)$ to $A$ defined by $\pi (\theta_{x, y}^X )=\theta_{x, y}^Y$, where
$x, y\in X$ and $\theta_{x, y}^X$ is the rank-one operator on $X$ defined as above. Then clearly $\pi$ is injective
and $\pi(\BB_B (X))=A_0$. Thus $A_0$ is closed and $A_0 =A$.

\begin{lemma}\label{lem:equivalent}With the above notations and assumptions, the inclusion
$A\subset C$ is unital and strongly Morita equivalent to the unital inclusion
$B\subset D$ with respect to $Y$ and its closed subspace $X$.
\end{lemma}
\begin{proof}By the above discussions, the inclusion $A\subset C$ is unital.
Clearly $A$ and $B$ are strongly Morita equivalent with respect to $X$ and $C$ and $D$ are
strongly Morita equivalent with respect to $Y$. For any $x, y, z\in Y$,
\begin{align*}
\theta_{x, y}^Y (z) & =x\cdot \la y, z \ra_D =x\cdot \la \sum_{i=1}^n x_i \cdot \la x_i , y \ra_D \, , \, z \ra_D
=\sum_{i=1}^n x\cdot \la y, x_i \ra_D \, \la x_i , z \ra_D \\
& =\sum_{i=1}^n \theta_{[x\cdot \la y, x_i \ra_D ], \, x_i }^Y (z) .
\end{align*}
Since $x_i \in X$, $[x\cdot \la y, x_i \ra_D ]\in Y$ for $i=1,2,\dots, n$,
$\theta_{x, y}^Y \in {}_C \la Y, X \ra$ for any $x, y\in Y$.
Thus ${}_C \la Y, X \ra=C$.
Therefore, $A\subset C$ is strongly Morita equivalent to $B\subset D$ with respect to
a $C-D$-equivalence bimodule $Y$ and its closed subspace $X$.
\end{proof}

Furthermore, we suppose that there is a conditional expectation $E^B$ of Watatani index-finite type from
$D$ onto $B$.

\begin{Def}\label{def:right}Let $E^X$ be a linear map from $Y$ onto $X$.
We say that $E^X$ is a
\sl
right conditional expectation
\rm
from $Y$ onto $X$ with respect to $E^B$ if $E^X$ satisfies the following conditions:
\newline
$(1)$ $E^X (x\cdot d)=x\cdot E^B (d)$ for any $d\in D$, $x\in X$,
\newline
$(2)$ $E^X (y\cdot b)=E^X (y)\cdot b$ for any $b\in B$, $y\in Y$,
\newline
$(3)$ $E^B(\la y, x\ra_D )=\la E^X (y), x \ra_D$ for any $x\in X$, $y\in Y$.
\end{Def}

\begin{remark}\label{remark:norm}(i) By Definition \ref {def:right}, we can see that
$E^B (\la y, x \ra_D )=\la E^X (y), x \ra_B$ for any $x\in X$, $y\in Y$.
\newline
(ii) $E^X$ is a projection of norm one from $Y$ onto
$X$. Indeed, by Raeburn and William \cite[the proof of Lemma 2.8]{RW:continuous},
for any $y\in Y$,
\begin{align*}
||E^X (y)|| & =\sup \{||\la E^X (y) , z \ra_B || \, | \, ||z||\le 1, \, z\in X \} \\
& =\sup \{||E^B ( \la y , z \ra_D ) || \, | \, ||z|| \le 1, z\in X \} \\
& \le \sup \{ ||y|| \, ||z|| \, \, | \, \, ||z|| \le 1, \, z\in X \} \\
& =||y|| .
\end{align*}
Since $E^X (x)=x$ for any $x\in X$, $E^X $ is a projection of norm one from $Y$ onto $X$.
\end{remark}

\begin{lemma}\label{lem:right_expectation}With the same assumptions as
in Lemma \ref{lem:equivalent}, we suppose that there is a conditional expectation
$E^B$ of Watatani index-finite type from $D$ onto $B$. Then there is a right conditional
expectation $E^X$ from $Y$ onto $X$ with respect to $E^B$.
\end{lemma}
\begin{proof}Let $E^X$ be the linear map from $Y$ to $X$ defined by
$$
\la E^X (y), x \ra_B =E^B (\la y, x \ra_D )
$$
for any $x\in X$, $y\in Y$. We show that Conditions (1), (2) in Definition \ref{def:right}
hold. Indeed, for any $x, y\in X$, $d\in D$,
$$
\la y, E^X (x\cdot d) \ra_B =E^B ( \la y, x\cdot d \ra_D )=E^B (\la y, x \ra_D d )=\la y, x \ra_B E^B (d)
= \la y, x\cdot E^B (d) \ra_B .
$$
Hence $E^X (x\cdot d)=x\cdot E^B (d)$ for any $x\in X$, $d\in D$. For any $b\in B$, $y\in Y$, $x\in X$,
\begin{align*}
\la x, E^X (y\cdot b) \ra_B & =E^B (\la x, y\cdot b \ra_D )=E^B ( \la x, y \ra_D b )=E^B (\la x, y \ra_D)b \\
& =\la x, E^X (y) \ra_B b =\la x, E^X (y)\cdot b \ra_B .
\end{align*}
Hence $E^X (y\cdot b)=E^X (y)\cdot b$ for any $y\in Y$, $b\in B$.
\end{proof}

\begin{lemma}\label{lem:right2}Let $A\subset C$ and $B\subset D$ be unital inclusions of unital
$C^*$-algebras, which are strongly Morita equivalent with respect to
a $C-D$-equivalence bimodule $Y$ and its closed subspace $X$.
Let $E^B$ be a conditional expectation of Watatani index-finite type from $D$ onto $B$
and $E^X$ a right conditional expectation from $Y$ onto $X$
with respect to $E^B$. Then for any $a\in A$, $y\in Y$, $E^X (a\cdot y)=a\cdot E^X (y)$.
\end{lemma}
\begin{proof}Since $X$ is full with the left $A$-valued inner product,
it suffices to show that
$$
E^X ({}_A \la x, z \ra \cdot y)={}_A \la x, z \ra \cdot E^X(y)
$$
for any $x, z\in X$, $y\in Y$. Indeed,
\begin{align*}
E^X ({}_A \la x, z \ra \cdot y) & =E^X (x \cdot \la z, y \ra_D )=x\cdot E^B ( \la z, y \ra_D )
=x \cdot \la z, E^X (y) \ra_B \\
& ={}_A \la x, z \ra \cdot E^X (y) .
\end{align*}
\end{proof}

\begin{prop}\label{prop:expectation2}With the same assumptions as in Lemma \ref{lem:right2},
there is a conditional expectation $E^A$ from $C$ onto $A$ such that
$E^X$ is a conditional expectation from $Y$ onto $X$ with respect to $E^A$ and $E^B$. 
\end{prop}
\begin{proof}Let $E^A$ be the linear map from $C$ onto $A$ defined by
$$
E^A (c) \cdot x=E^X (c\cdot x)
$$
for any $c\in C$, $x\in X$. First, we note that Conditions in Definition \ref{def:expectation} except Condition (3)
hold by the assumptions and Lemma \ref{lem:right2}. We show that Condition (3) in Definition \ref{def:expectation}
holds. Indeed fot any $x, z\in X$, $y\in Y$,
$$
E^A ({}_C \la y, x \ra )\cdot z =E^X ({}_C \la y, x \ra \cdot z) =E^X (y \cdot \la x, z \ra_B )
=E^X (y) \cdot \la x, z \ra_B ={}_C \la E^X (y), x \ra \cdot z .
$$
Hence for any $x\in X$, $y\in Y$, $E^A ({}_C \la y, \, x \ra )={}_C \la E^X (y), \, x \ra $.
Next, we show that $E^A$ is a conditional expectation from $C$ onto $A$.
For any $a\in A$, $x\in X$,
$$
E^A (a)\cdot x=E^X (a\cdot x)=a\cdot E^X(x)=a\cdot x
$$
by Lemma \ref{lem:right2}. Hence $E^A (a)=a$ for any $a\in A$.
For any $c\in C$, $x\in X$,
$$
||E^A (c)\cdot x||=||E^X (c\cdot x)||\le ||c\cdot x||\le ||c||\, ||x||
$$
by Remark \ref{remark:norm} (ii). Hence $||E^A ||=1$ since $E^A (a)=a$ for any $a\in A$.
Thus $E^A$ is a projection of norm one from $C$ onto $A$.
It follows by Tomiyama \cite [Theorem 1]{Tomiyama:projection} that
$E^A$ is a conditional expectation from $C$ onto $A$.
Therefore, we obtain the conclusion.
\end{proof}

Let $B\subset D$ be a unital inclusion of unital $C^*$-algebras and let $Y$ be
a full right Hilbert $D$-module and $X$ its closed subspace satisfying Conditions
(1)-(3) in the beginning of this section. We suppose that there is a conditional
expectation $E^B$ of Watatani index-finite type from $D$ onto $B$.
Let $C=\BB_D (Y)$ and let $A$ be the $C^*$-subalgebra, the linear spans of the set
$\{\theta_{x, y}^Y \, | \, x, y\in X \}$. Then by Lemmas \ref{lem:equivalent},
\ref{lem:right_expectation}, \ref{lem:right2} and Proposition \ref{prop:expectation2},
there are a conditional expectation $E^X$ from $Y$ onto $X$ and a conditional
expectation $E^A$ from $C$ onto $A$ such that $E^X$ is a conditional expectation from $Y$ onto $X$
with respect to $E^A$ and $E^B$. We note that a conditional expectation $E^A$ is depend only on
$E^B$ and $E^X$ by Condition (3) in Definition \ref{def:expectation}.
Hence by Theorem \ref{thm:morita}, $E^A$ is of Watatani index-finite type.
Thus we obtain the following corollary:

\begin{cor}\label{cor:finite}With the same notations as in Proposition \ref{prop:expectation2},
a conditional expectation $E^A$ from $C$ onto $A$ defined in Proposition \ref{prop:expectation2}
is of Watatani index-finite type.
\end{cor}

Combining the above results, we obtain the following:

\begin{thm}\label{thm:combine}Let $B\subset D$ be a unital inclusion of
unital $C^*$-algebras and let $Y$ be a full right Hilbert $D$-module and
$X$ its closed subspace satisfying Conditions (1)-(3) in the beginning of this
section. Let $E^B$ be a conditional expectation of Watatani index-finite type from $D$ onto $B$.
Let $C=\BB_D (Y)$ and let $A$ be the $C^*$-subalgebra, the linear spans of the set
$\{\theta_{x, y}^Y \, | \, x, y\in X \}$.
Then there are a conditional expectation $E^A$ of Watatani index-finite type from $C$ onto $A$ and
a conditional expectation $E^X$ from $Y$ onto $X$ with respect to $E^A$ and $E^B$.
\end{thm}

\begin{remark}\label{remark:left}(i) In the same way as in Definition \ref{def:right},
we can define a left conditional expectation in the following situation: Let $A\subset C$ be a unital
inclusion of unital $C^*$-algebras and let $Y$ be a full left Hilbert $C$-module and $X$ its
closed subspace satisfying that
\newline
(1) $a\cdot x\in X$, ${}_C \la x, y \ra \in A$ for any $a\in A$, $x, y\in X$,
\newline
(2) $\overline{{}_C \la X, X \ra}=A$, $\overline{{}_C \la Y, X \ra}=C$,
\newline
(3) There is a finite set $\{x_i \}_{i=1}^n \subset Y$ such that for any $y\in Y$
$$
\sum_{i=1}^n {}_C \la y, x_i \ra \cdot x_i =y .
$$
We note that $Y$ is of finite type and that $X$ can be regarded as a full left Hilbert $A$-module
of finite type in the sense of Kajiwara and Watatani \cite{KW1:bimodule}.
\newline
(ii) A conditional expectation from an equivalence onto its closed subspace in
Definition \ref{def:expectation} is a left and right conditional expectation.
\newline
(iii) We have the results on a left conditional expectation similar to the above.
\end{remark}

\section{Examples}\label{sec:example}In this section, we shall give two
examples of conditional expectations from equivalence bimodules onto
their closed subspaces.
\par
First, let $A$ and $B$ be unital $C^*$-algebras which are strongly
Morita equivalent with respect to an $A-B$-equivalence bimodule $X$.
Let $H$ be a finite dimensional $C^*$-Hopf algebra with its dual $C^*$-Hopf
algebra $H^0$. Let $\rho$ and $\sigma$ be coactions of $H^0$ on $A$ and $B$,
respectively. We suppose that $\rho$ and $\sigma$ are strongly Morita
equivalent with respect to a coaction $\lambda$ of $H^0$ on $X$, respectively,
that is, $(A, B, X, \rho, \sigma, \lambda, H^0 )$ is a covariant system (See \cite {KT3:equivalence}).
We use the same notations as in \cite {KT3:equivalence}.
Let
$$
C=A\rtimes_{\rho}H, \quad D=B\rtimes_{\sigma}H
$$
be crossed products of $C^*$-algebras $A$ and $B$ by the actions of the finite dimensional
$C^*$-Hopf algebra $H$ induced by $\rho$ and $\sigma$, respectively.
Also, let $Y=X\rtimes_{\lambda}H$ be the
crossed product of an $A-B$-equivalence bimodule $X$ by the action of $H$ induced by $\lambda$.
Then by \cite [Corollary 4.7]{KT3:equivalence}, $Y$ is a $C-D$-equivalence bimodule and
$C$ and $D$ are strongly Morita equivalent with respect to $Y$. We can see that the unital inclusion
$A\subset C$ and $B\subset D$ are strongly Morita
equivalent with respect to $Y$ and its closed subspace $X$ by easy computations.
Indeed, it suffices to show that ${}_C \la X, Y \ra =C$ and $\la X, Y \ra_D =D$ since the other conditions
in Definition \ref{def:inclusion} clearly hold. For any $x, y\in X$, $h\in H$,
\begin{align*}
{}_C \la x\rtimes_{\lambda}1 \, , (1\rtimes_{\rho}h)^*(y\rtimes_{\lambda}1) \ra & =
((1\rtimes_{\rho}h)^* {}_C \la y\rtimes_{\lambda}1 , x\rtimes_{\rho}1 \ra)^* \\
& ={}_C \la x\rtimes_{\lambda}1 , \, y\rtimes_{\lambda} 1\ra (1\rtimes_{\rho}h)
={}_A \la x , y \ra\rtimes_{\rho}h .
\end{align*}
Hence ${}_C \la X, \, Y \ra =C$. Also,
$$
\la x\rtimes_{\lambda}1 \, , y\rtimes_{\lambda}h \ra_{D}
=\la x, y \ra_B\rtimes_{\sigma}h .
$$
Thus $\la X, Y \ra_D =D$.

Let $E_1^{\rho}$ and $E_1^{\sigma}$ be the canonical conditional expectations from
$A\rtimes_{\rho}H$ and $B\rtimes_{\sigma}H$ onto $A$ and $B$ defined by
$$
E_1^{\rho}(a\rtimes_{\rho}h)=\tau(h)a, \quad E_1^{\sigma}(b\rtimes_{\sigma}h)=\tau(h)b
$$
for any $a\in A$, $b\in B$, $h\in H$, respectively, where $\tau$ is the Haar trace
on $H$. Let $E_1^{\lambda}$ be the linear map from $X\rtimes_{\lambda}H$ onto $X$
defined by
$$
E_1^{\lambda}(x\rtimes_{\lambda}h)=\tau(h)x
$$
for any $x\in X$, $h\in H$.

\begin{prop}\label{prop:morita2}With the above notations, $E_1^{\lambda}$ is a conditional
expectation from $X\rtimes_{\lambda}H$ onto $X$ with respect to $E^A$ and $E^B$.
\end{prop}
\begin{proof}
Let $X, Y$ and $E_1^{\lambda}$ be as above.
We claim that $E_1^{\rho}$, $E_1^{\sigma}$ and $E_1^{\lambda}$ satisfy
Conditions (1)-(6) in Definition \ref{def:expectation}. Indeed, we compute the following:
\newline
(1) For any $a\in A$, $x\in X$, $h\in H$,
\begin{align*}
E_1^{\lambda}((a\rtimes_{\rho}h)\cdot (x\rtimes_{\lambda}1)) & =
E_1^{\lambda}(a\cdot [h_{(1)}\cdot_{\lambda}x]\rtimes_{\lambda}h_{(2)}) \\
& =a\cdot x\tau(h)\rtimes_{\lambda}1 =E_1^{\rho}(a\rtimes_{\rho}h)\cdot (x\rtimes_{\lambda}1) .
\end{align*}
(2) For any $a\in A$, $x\in X$, $h\in H$,
$$
E_1^{\lambda}((a\rtimes_{\rho}1)\cdot (x\rtimes_{\lambda}h)) 
=E_1^{\lambda}(a\cdot x\rtimes_{\lambda}h)=\tau(h)a\cdot x\rtimes_{\lambda}1
=(a\rtimes_{\rho}1)\cdot E_1^{\lambda}(x\rtimes_{\lambda}h) .
$$
(3) For any $x, y\in X$, $h\in H$,
\begin{align*}
E_1^{\rho}({}_C\la y\rtimes_{\lambda}h , \, x\rtimes_{\lambda}1 \ra) & =
E_1^{\rho}({}_A \la y, \, [S(h_{(1)})^* \cdot_{\lambda} x] \ra \rtimes_{\rho}h_{(2)}) \\
& ={}_A \la y, \, [S(h_{(1)})^* \cdot_{\lambda} x]\ra\tau(h_{(2)}) \\
& ={}_A \la y, \, \overline{\tau(h)}x \ra
={}_A \la E_1^{\lambda}(y\rtimes_{\lambda}h), \, x \ra .
\end{align*}
(4) For any $b\in B$, $x\in X$, $h\in H$,
$$
E_1^{\lambda}((x\rtimes_{\lambda}1)\cdot (b\rtimes_{\sigma}h)) 
=E_1^{\lambda}(x\cdot b\rtimes_{\lambda}h)=\tau(h)(x\cdot b\rtimes_{\lambda}1)
=(x\rtimes_{\lambda}1)\cdot E_1^{\sigma}(b\rtimes_{\sigma}h) .
$$
(5) For any $b\in B$, $x\in X$, $h\in H$,
\begin{align*}
E_1^{\lambda}((x\rtimes_{\lambda}h)\cdot (b\rtimes_{\sigma}1)) & =
E_1^{\lambda}(x\cdot [h_{(1)}\cdot_{\sigma}b]\rtimes_{\lambda}h_{(2)})
=x\cdot b\tau(h)\rtimes_{\lambda}1 \\
& =E_1^{\lambda}(x\rtimes_{\lambda}h)\cdot (b\rtimes_{\sigma}1) .
\end{align*}
(6) For any $x, y\in X$, $h\in H$,
\begin{align*}
E_1^{\sigma}(\la y\rtimes_{\lambda}h , \, x\rtimes_{\lambda}1 \ra_D ) & =
E_1^{\sigma}([h_{(1)}^* \cdot_{\sigma} \la y, \, x \ra_B ]\rtimes_{\sigma}h_{(2)}^* ) \\
& =\tau (h^* ) \la y, \, x \ra_B =\la E_1^{\lambda} (y\rtimes_{\lambda}h ), \, x\rtimes_{\lambda}1 \ra_B .
\end{align*}
Therefore, we obtain the conclusion.
\end{proof}

We shall give another example. Let $A\subset B$ be a unital inclusion of unital $C^*$-algebras
and let $F$ be a conditional expectation of Watatani index-finite type from $B$ onto $A$.
Let $f$ be the Jones projection and $B_1$ the $C^*$-basic construction for $F$. Let $F_1$ be
its dual conditional expectation from $B_1$ onto $B$. Let $f_1$ be the Jones projection and $B_2$
the $C^*$-basic construction for $F_1$. Let $F_2$ be the dual conditional expectation of $F_1$
from $B_2$ onto $B_1$. Then $A$ is strongly Morita equivalent to $B_1$ and $B$ is
strongly Morita equivalent to $B_2$ by Watatani \cite {Watatani:index}.
Since $F$ and $F_1$ are of Watatani index-finite type, $B$ and $B_1$ can be equivalence bimodules,
that is, $B$ can be regarded as a $B_1 -A$-equivalence bimodule as follows:
For any $a\in A$, $x, y, z\in B$,
$$
{}_{B_1} \la x, \, y \ra =xfy^* , \quad \la x, \, y \ra_A =F(x^* y), \qquad
xfy\cdot z =xF(yz), \quad  x\cdot a=xa .
$$
Also, $B_1$ can be regarded as a $B_2 -B$-equivalence bimodule as follows:
For any $b\in B$, $x, y, z\in B_1$,
$$
{}_{B_2} \la x, \, y \ra =xf_1 y^* , \quad \la x, \, y \ra_B =F_1 (x^* y) , \qquad
xf_1 y\cdot z=xF_1 (yz), \quad x\cdot b=xb .
$$
We denote by $\Ind_W (F)$ the Watatani index of a conditional expectation $F$ from
$B$ onto $A$. Also, let $\{(w_i , \, w_i^* )\}_{i=1}^n$ be a quasi-basis for $F_1$.

\begin{lemma}\label{lem:equivalent2}With the above notations, we suppose that $\Ind_W (F)\in A$.
Then the inclusions $A\subset B$ and $B_1 \subset B_2$ are strongly Morita equivalent.
\end{lemma}
\begin{proof}Let $\theta$ be the linear map from $B$ to $B_1$ defined by
$$
\theta(x)=\Ind_W (F)^{\frac{1}{2}}xf
$$
for any $x\in B$. Then for any $a\in A$, $x, y, z\in B$,
$$
\theta(xfy\cdot z\cdot a)=\theta(xF(yz)a)=\Ind_W (F)^{\frac{1}{2}}xF(yz)af
=\Ind_W (F)^{\frac{1}{2}}xF(yz)fa .
$$
On the other hand, since $\Ind_W (F)\in A\cap B'$,
\begin{align*}
xfy\cdot \theta(z)\cdot a & =xfy\cdot \Ind_W (F)^{\frac{1}{2}}zf\cdot a 
=\sum_{i=1}^n xfyw_i f_1 w_i^* \cdot \Ind_W (F)^{\frac{1}{2}}zf\cdot a \\
& =xfy\Ind_W (F)^{\frac{1}{2}}zfa =xF(y\Ind_W (F)^{\frac{1}{2}}z)fa 
=\Ind_W (F)^{\frac{1}{2}}xF(yz)fa .
\end{align*}
Thus $\theta$ is a $B_1 -A$-bimodule map. Furthermore, for any $x, y\in B$,
\begin{align*}
\la \theta(x), \, \theta(y) \ra_B & =F_1(\theta(x)^* \theta(y))=F_1 ((\Ind_W (F)^{\frac{1}{2}}xf)^* (\Ind_W (F)^{\frac{1}{2}}yf )) \\
& =\Ind_W (F)F_1 (fx^* yf)=\Ind_W (F)F_1 (F(x^* y)f) =F(x^* y ) \\
& =\la x, \, y \ra_A , \\
{}_{B_2} \la \theta(x), \, \theta(y) \ra & =\theta(x)f_1 \theta(y)^* =\Ind_W (F)xff_1 fy^* =xfy^* ={}_{B_1} \la x, \, y \ra
\end{align*}
by \cite [Lemma 2.3.5]{Watatani:index}. Thus we regard $B$ as a closed subspace of the
$B_2 -B$-equivalence bimodule $B_1$ by the map $\theta$. In order to obtain the conclusion,
it suffices to show that ${}_{B_2} \la B, B_1 \ra =B_2$ and $\la B , B_1 \ra_B =B$
since the other conditions in Definition \ref{def:inclusion} clearly
hold. Let $x, y, z\in B$. Then
$$
{}_{B_2} \la x, yfz \ra ={}_{B_2} \la \theta(x), yfz \ra ={}_{B_2}\la \Ind_W (F)^{\frac{1}{2}}xf, yfz \ra
=\Ind_W (F)^{\frac{1}{2}}xff_1 z^* fy .
$$
Since $f_1 z^* =z^* f_1$, ${}_{B_2} \la B, B_1 \ra =B_2$. Also,
\begin{align*}
\la x, yfz \ra_B & = \la \theta(x), yfz \ra_B =\la \Ind_W (F)^{\frac{1}{2}}xf, yfz \ra_B
=F_1 (\Ind_W (F)^{\frac{1}{2}}fx^* yfz) \\
& =F_1 (\Ind_W (F)^{\frac{1}{2}}F(x^* y)fz)=\Ind_W (F)^{-\frac{1}{2}}F(x^* y)z .
\end{align*}
Hence $\la B, B_1 \ra_B =B$. Therefore, we obtain the conclusion.
\end{proof}

\begin{prop}\label{prop:morita3}With the above notations, we regard $B$ as a closed subspace
of $B_2$ by the linear map $\theta$ defined in Lemma \ref{lem:equivalent2} and we suppose that $\Ind_W (F)\in A$.
Then there is a conditional expectation $G$ from $B_1$ onto $B$ with respect to
$F$ and $F_2$.
\end{prop}
\begin{proof}Let $G$ be the linear map from $B_1$ onto $B$ defined by
$$
G(xfy)=xF(y)f=\theta (\Ind_W (F)^{-\frac{1}{2}}xF(y))
$$
for any $x, y\in B$, where we identify $\theta (\Ind_W (F)^{-\frac{1}{2}}xF(y))$
with $\Ind_W (F)^{-\frac{1}{2}}xF(y)$.
By routine computations, we can see that $G$ satisfies Conditions (1)-(6) in
Definition \ref{def:expectation}. Indeed, we compute the following:
\newline
(1) For any $x_1 =afb$, $y_1 =a_1 fb_1 \in B_1$, $a, b, a_1 , b_1 \in B$ and
$z\in B$,
\begin{align*}
G(x_1 f_1 y_1 \cdot \theta(z)) & =G(x_1 f_1 y_1 \cdot \Ind_W (F)^{\frac{1}{2}}zf ) 
=G(x_1 F_1(y_1 \Ind_W (F)^{\frac{1}{2}}zf)) \\
& =G(afbF_1 (a_1 fb_1 \Ind_W (F)^{\frac{1}{2}}zf)) \\
& =G(\Ind_W (F)^{\frac{1}{2}}afbF_1 (a_1 F(b_1 z)f)) \\
& =\Ind_W (F)^{-\frac{1}{2}}aF(ba_1 F(b_1 z))f \\
& =\Ind_W (F)^{-\frac{1}{2}}aF(ba_1 )F(b_1 z)f .
\end{align*}
On the other hand,
\begin{align*}
F_2 (x_1 f_1 y_1 )\cdot z  & =\Ind_W (F)^{-1}x_1 y_1 \cdot z
=\Ind_W (F)^{-1}afba_1 f b_1 \cdot z \\
& =\Ind_W (F)^{-1}aF(ba_1 )fb_1 \cdot z
=\Ind_W (F)^{-1}aF(ba_1 )F(b_1 z) .
\end{align*}
Since we identify $\theta(\Ind_W (F)^{-1}aF(ba_1 )F(b_1 z))$ with $\Ind_W (F)^{-\frac{1}{2}}aF(ba_1 )F(b_1 z)f$,
we can see that $G$ satisfies Condition (1) in Definition \ref{def:expectation}.
\newline
(2) For any $a, b, x, y\in B$,
$$
G(afb\cdot xfy)=G(afbxfy)=G(aF(bx)fy)=\theta(\Ind_W (F)^{-\frac{1}{2}}aF(bx)F(y)) .
$$
On the other hand,
\begin{align*}
afb\cdot G(xfy) & =afb\cdot\Ind_W (F)^{-\frac{1}{2}}xF(y)=aF(b\Ind_W (F)^{-\frac{1}{2}}xF(y)) \\
& =\Ind_W (F)^{-\frac{1}{2}}aF(bx)F(y) .
\end{align*}
Thus $G$ satisfies Condition (2) in Definition \ref{def:expectation}.
\newline
(3) For any $x, y, z\in B$,
$$
{}_{B_2} \la G(xfy), \, \theta(z) \ra ={}_{B_2} \la xF(y)f, \, \Ind_W (F)^{\frac{1}{2}}zf \ra
=\Ind_W (F)^{-\frac{1}{2}}xF(y)fz^* .
$$
On the other hand,
\begin{align*}
F_2 ({}_{B_2} \la xfy, \, \theta(z) \ra ) & =F_2 ({}_{B_2} \la xfy, \, \Ind_W (F)^{\frac{1}{2}}zf \ra )
=F_2 (xfyf_1 fz^* \Ind_W (F)^{\frac{1}{2}}) \\
& =\Ind_W (F)^{-\frac{1}{2}}xfyfz^* =\Ind_W (F)^{-\frac{1}{2}}xF(y)fz^* .
\end{align*}
Thus $G$ satisfies Condition (3) in Definition \ref{def:expectation}.
\newline
(4) For any $b, z\in B$,
$$
G(\theta(z)\cdot b)=G(\Ind_W (F)^{\frac{1}{2}}zf\cdot b)=G(\Ind_W (F)^{\frac{1}{2}}zfb)
=\Ind_W (F)^{\frac{1}{2}}zF(b)f .
$$
On the other hand,
$$
\theta(z)\cdot F(b)=\Ind_W (F)^{\frac{1}{2}}zfF(b)=\Ind_W (F)^{\frac{1}{2}}zF(b)f .
$$
Thus $G$ satisfies Condition (3) in Definition \ref{def:expectation}.
\newline
(5) For any $a\in A$, $x, y\in B$,
$$
G(a\cdot xfy)=G(axfy)=axF(y)f=a\cdot G(xfy).
$$
Thus $G$ satisfies Condition (5) in Definition \ref{def:expectation}.
\newline
(6) For any $x, y, z\in B$,
\begin{align*}
F(\la xfy, \, \theta(z) \ra_B ) & =F(F_1 (y^* fx^* \Ind_W (F)^{\frac{1}{2}}zf ))
=F(F_1 (y^* F(x^* z)\Ind_W (F)^{\frac{1}{2}}f)) \\
& =\Ind_W (F)^{-\frac{1}{2}}F(y^* F(x^* z))
=\Ind_W (F)^{-\frac{1}{2}}F(y^* )F(x^* z) .
\end{align*}
On the other hand,
$$
\la G(xfy),\, \theta(z) \ra_B = \la xF(y)f, \, \Ind_W (F)^{\frac{1}{2}}zf \ra_B
=\Ind_W (F)^{-\frac{1}{2}}F(y^* )F(x^* z) .
$$
Thus $G$ satisfies Condition (6) in Definition \ref{def:expectation}.
Therefore, we obtain the conclusion.
\end{proof}

\section{Linking algebras and conditional expectations}\label{sec:linking}
Let $A\subset C$ and $B\subset D$ be unital inclusions of unital $C^*$-algebras,
which are strongly Morita equivalent with respect to a $C-D$-equivalence
bimodule $Y$ and its closed subspace $X$. We regard $Y$ and $X$ as a full right
Hilbert $D$-module and its closed subspace, respectively. Then $Y$ and $X$ satisfy Conditions 
at the beginning of Section \ref{sec:one-sided}. 
We also note that the full right Hilbert $D$-module $ Y\oplus D$
and its closed subspace $X\oplus B$ satisfy Conditions at the beginning of Section \ref{sec:one-sided}.
Let $L_X =\BB_B (X\oplus B)$ and $L_Y =\BB_D (Y\oplus D)$.
By Raeburn and Williams
\cite [Corollary 3.21]{RW:continuous}, $L_X$ and $L_Y$ are isomorphic to
the linking algebras induced by equivalence bimodules $X$ and $Y$, respectively.
We denote the linking algebras by the same symbols $L_X$ and $L_Y$, respectively.
In the same way as in the proof of Brown, Green and Rieffel \cite [Theorem 1.1]{BGR:linking},
we obtain the following proposition:

\begin{prop}\label{prop:corner}Let $A\subset C$ and $B\subset D$ be unital
inclusions of unital $C^*$-algebras. Then the inclusions $A\subset C$ and $B\subset D$ are
strongly Morita equivalent if and only if there is a unital inclusion of unital $C^*$-algebras
$K\subset L$ and projections in $K$ satisfying that
\newline
$(1)$ $pKp\cong A$, $pLp\cong C$,
\newline
$(2)$ $qKq\cong B$, $qLq\cong D$,
\newline
$(3)$ $KpK=KqK=K$, $LpL=LqL=L$, $p+q=1_L$.
\end{prop}

We suppose that there is a conditional expectation $E^B$ of Watatani index-finite type from $D$
onto $B$. By Lemma \ref{lem:right_expectation},
there is a right conditional expectation $E^X$ from $Y$ onto $X$ with respect to
$E^B$.

\begin{lemma}\label{lem:sum}The linear map $E^X \oplus E^B$ is a right
conditional expectation from $Y\oplus D$ onto $X\oplus B$ with respect to
$E^B$.
\end{lemma}
\begin{proof}
We show that Conditions (1)-(3) in Definition \ref{def:right} hold.
\newline
(1) For any $x\in X$, $b\in B$, $d\in D$,
$$
(E^X \oplus E^B )((x\oplus b)\cdot d)=(E^X \oplus E^B )((x\cdot d)\oplus bd)
=x\cdot E^B (d)\oplus bE^B (d)=(x\oplus b)\cdot E^B (d) .
$$
(2) For any $b\in B$, $y\in Y$, $d\in D$,
$$
(E^X \oplus E^B )((y\oplus d) \cdot b)=(E^X \oplus E^B )((y\cdot b)\oplus db)=(E^X (y)\oplus d)\cdot b .
$$
(3) For any $x\in X$, $b\in B$, $y\in Y$, $d\in D$,
\begin{align*}
\la (E^X \oplus E^B )(y\oplus d), \, x\oplus b \ra_D & =\la E^X (y)\oplus E^B (d), \, x\oplus b \ra_D \\
& =\la E^X (y), \, x \ra_D +E^B (d)^* b \\
& =E^B ( \la y, \, x \ra_D )+E^B (d^* b) \\
& =E^B ( \la y\oplus d, \, x\oplus b \ra_D ) .
\end{align*}
Therefore, Conditions (1)-(3) in Definition \ref{def:right} hold.
\end{proof}

By Proposition \ref{prop:expectation2} and Corollary \ref{cor:finite},
there is a conditional expectation $E^{L_X}$ of Watatani index-finite type from $L_Y$ onto $L_X$
such that $E^X \oplus E^B$ is a conditional expectation from $Y\oplus D$ onto
$X\oplus B$ with respect to $E^{L_X}$ and $E^B$.
Since we identify $L_X$ and $L_Y$ with
the linking algebras induced by equivalence bimodules $X$ and $Y$, respectively,
we obtain the following proposition:

\begin{prop}\label{prop:linking}With the above notations, we can write
$$
E^{L_X}(\begin{bmatrix} c & x \\
\widetilde{y} & d \end{bmatrix})=\begin{bmatrix} E^A (c) & E^X (x) \\
\widetilde{E^X (y)} & E^B (d) \end{bmatrix}
$$
for any element $\begin{bmatrix} c & x \\
\widetilde{y} & d \end{bmatrix}\in L_Y$, where for any $z\in X$,
we denote by $\widetilde{z}$ its corresponding element in $\widetilde{X}$, the dual Hilbert $C^*$-bimodule of $X$.
\end{prop}
\begin{proof}Let $\theta_{y\oplus d, z\oplus f}$ be the rank-one operator on $Y\oplus D$ induced
by $y\oplus d, z\oplus f \in Y\oplus D$. Then by Definition \ref{def:expectation},
for any $x\oplus b\in X\oplus B$,
\begin{align*}
E^{L_X} (\theta_{y\oplus d, z\oplus f})\cdot (x\oplus b) & =(E^X \oplus E^B )(\theta_{y\oplus d, z\oplus f} (x\oplus b)) \\
& =(E^X \oplus E^B )(y\oplus d \cdot \la z\oplus f, \, x\oplus d \ra_D ) \\
& =(E^X \oplus E^B )(y\oplus d\cdot (\la z, x \ra _D +f^* b )) \\
& =E^X (y\cdot (\la z, x \ra_D +f^* b ))\oplus E^B (d(\la z, x \ra_D +f^* b )) .
\end{align*}
On the other hand, since we identify $L_X$ and $L_Y$ with the linking algebras
induced by $X$ and $Y$, respectively, by the proof of \cite [Corollary 3.21]{RW:continuous},
we regard $\theta_{y\oplus d, z\oplus f}$ as an element
$\begin{bmatrix} {}_C \la y, z \ra & y\cdot f^* \\
\widetilde{z\cdot d^* } & df^* \end{bmatrix}$. Then
\begin{align*}
\begin{bmatrix}E^A ({}_C \la y , z \ra ) & E^X (y\cdot f^* ) \\
\widetilde{E^X (z\cdot d^* )} & E^B (df^* ) \end{bmatrix}
\left[
\begin{array}{ccc}
x \\
b
\end{array} \right]
& =\left[ \begin{array}{ccc}
E^A ({}_C \la y, z \ra) \cdot x+E^X (y\cdot f^* )\cdot b \\
\la E^X (z\cdot d^* ), \, x \ra_D +E^B (df^* )b \end{array}\right] \\
& =\left[ \begin{array}{ccc} E^X ({}_C \la y, \, z \ra \cdot x+y\cdot f^* b ) \\
E^B (\la z\cdot d^*, \, x \ra_D +df^* b ) \end{array} \right] \\
& =E^{L_X}(\theta_{y\oplus d, z\oplus f})\cdot (x\oplus b ) .
\end{align*}
Therefore, we obtain the conclusion.
\end{proof}

\begin{lemma}\label{lem:basis}With the above notations, let $\{(u_i, \, u_i^* )\}_{i=1}^n$
and $\{(v_j, \, v_j^* )\}_{j=1}^m$ be any quasi-bases for $E^A$ and $E^B$, respectively. Then
for any $y\in Y$,
$$
y=\sum_{j=1}^m E^X (y\cdot v_j )\cdot v_j^* =\sum_{i=1}^n u_i \cdot E^X (u_i^* \cdot y) .
$$
\end{lemma}
\begin{proof}
By the discussions in Section \ref{sec:definition}, we may assume the following:
$$
B=pM_k (A)p, \quad D=pM_k (C)p, \quad X=(1\otimes f)M_k (A)p, \quad Y=(1\otimes f)M_k (C)p,
$$
where $k$ is a positive integer,
$f=\begin{bmatrix} 1 & 0 & \ldots & 0 \\
0 & 0 & \ldots & 0 \\
\vdots & \vdots & \ddots & \vdots \\
0 & 0 & \ldots & 0 \end{bmatrix}_{k\times k}$ and $p$ is a full projection in $M_k (A)$.
Furthermore, we regard $X$ and $Y$ as
an $A-pM_k (A)p$-equivalence bimodule and a $C-pM_k (C)p$-equivalence bimodule
in the usual way. Also, we can suppose that
$$
E^B =(E^A \otimes\id_{M_k (\BC)}) |_{pM_k (C)p}, \quad
E^X =(E^A \otimes\id_{M_k (\BC)}) |_{(1\otimes f)M_k (C)p} ,
$$
respectively. Let $\{(u_i ,\, u_i^* )\}_{i=1}^n$ be any quasi-basis for $E^A$. For any $c\in C$,
$h\in M_k (\BC)$,
\begin{align*}
\sum_{i=1}^n u_i \cdot E^X (u_i^* \cdot (1\otimes f)(c\otimes h)p) & =
\sum_{i=1}^n u_i \cdot (E^A \otimes\id_{M_k (\BC)})((u_i^* \otimes f)(c\otimes h)p) \\
& =\sum_{i=1}^n u_i \cdot (E^A (u_i^* c)\otimes fh)p \\
& =\sum_{i=1}^n (u_iE^A (u_i^* c)\otimes fh)p \\
& =\sum_{i=1}^n (c\otimes fh)p=(1\otimes f)(c\otimes h)p .
\end{align*}
Replacing the left hand side by the right hand side, in the similar way to the above, we
can obtain the other equation.
\end{proof}

\begin{lemma}\label{lem:index}With the above notations, for any $y\in Y$,
$$
\Ind_W (E^A )\cdot y =y\cdot \Ind_W (E^B ) .
$$
\end{lemma}
\begin{proof}
By Lemma \ref{lem:basis}, for any $y\in Y$,
$$
\sum_{i,j}u_i \cdot E^X (u_i^* \cdot y\cdot v_j )\cdot v_j^*
=\sum_j y\cdot v_j v_j^* =y\cdot\Ind_W (E^B ) .
$$
Similarly 
$$
\sum_{i, j}u_i \cdot E^X (u_i^* \cdot y\cdot v_j )\cdot v_j^* =\Ind_W (E^A )\cdot y .
$$
Hence, we obtain the conclusion.
\end{proof}

\begin{cor}\label{cor:quasi-basis}With the above notations,
$$
\{( \begin{bmatrix} u _i & 0 \\
0 & v_j \end{bmatrix} \, , \, \begin{bmatrix} u _i & 0 \\
0 & v_j \end{bmatrix}^* ) \, | \, i=1,2,\dots, n, \, j=1,2,\dots, m \}
$$
is a quasi-basis for $E^{L_X}$ and $\Ind_W (E^{L_X})=\begin{bmatrix} \Ind_W (E^A ) & 0 \\
0 & \Ind_W (E^B ) \end{bmatrix}$.
\end{cor}
\begin{proof}By Lemma \ref{lem:basis} and routine computations, we can see that
$$
\{( \begin{bmatrix} u _i & 0 \\
0 & v_j \end{bmatrix} \, , \, \begin{bmatrix} u _i & 0 \\
0 & v_j \end{bmatrix}^* ) \, | \, i=1,2,\dots, n, \, j=1,2,\dots, m \}
$$
is a quasi-basis for $E^{L_X}$. Hence by the definition of Watatani index, we can see that
$\Ind_W (E^{L_X})=\begin{bmatrix} \Ind_W (E^A ) & 0 \\
0 & \Ind_W (E^B ) \end{bmatrix}$.
\end{proof}

\section{The upward basic construction}\label{sec:upward}
Let $A\subset C$ and $B\subset D$ be unital inclusions of unital $C^*$-algebras,
which are strongly Morita equivalent with respect to a $C-D$-equivalence bimodule $Y$
and its closed subspace $X$. We suppose that there are conditional expectations $E^A$
and $E^B$ from $C$ and $D$ onto $A$ and $B$, which are of Watatani index-finite type,
respectively. Also, we suppose that there is a conditional expectation $E^X$ from
$Y$ onto $X$ with respect to $E^A$ and $E^B$.
Let $e_A$ and $e_B$ be the Jones projections for $E^A$ and $E^B$, respectively
and let $C_1$ and $D_1$ be the $C^*$-basic constructions for $E^A$ and $E^B$,
respectively. We regard $C$ and $D$ as a $C_1 -A$-equivalence bimodule and
a $D_1 -B$-equivalence bimodule in the same way as in Section \ref{sec:example}. Let
$$
Y_1 =C\otimes_A X\otimes_B \widetilde{D} ,
$$
where $\widetilde{D}$ is the dual equivalence bimodule of $D$,
a $B-D_1$-equivalence bimodule. Clearly $Y_1$ is a $C_1 -D_1$-equivalence
bimodule. Let $E^Y$ be the linear map from $Y_1$ to $Y$ defined by
$$
E^Y (c\otimes x\otimes \widetilde{d})=\Ind_W (E^A )^{-1}c\cdot x\cdot d^*
$$
for any $c\in C$, $d\in D$, $x\in X$. Then $E^Y$ is well-defined, clearly.
For any $y\in Y$,
$$
E^Y (\sum_{i=1}^n u_i \otimes E^X (u_i^* \cdot y)\otimes\widetilde{1})
=\sum_{i=1}^n \Ind_W (E^A )^{-1}u_i \cdot E^X (u_i^* \cdot y)
=\Ind_W (E^A )^{-1}\cdot y
$$
by Lemma \ref{lem:basis}. Hence $E^Y$ is surjective. Also, we note that
$$
E^Y (c\otimes x\otimes \widetilde{d}) =\Ind_W (E^A )^{-1}c\cdot x\cdot d^*
=c\cdot x\cdot d^* \Ind_W (E^B )^{-1}
$$
for any $c\in C$, $d\in D$, $x\in X$ by Lemma \ref{lem:index}.
Let $\phi$ be the linear map from $Y$ to $Y_1$ defined by
$$
\phi(y)=\sum_{i,j}u_i \otimes E^X(u_i^* \cdot y\cdot v_j )\otimes \widetilde{v_j}
$$
for any $y\in Y$.

\begin{lemma}\label{lem:bimodule}With the above notations, we have following conditions:
For any $c\in C$, $d\in D$, $y, z\in Y$,
\newline
$(1)$ $\phi(c\cdot y)=c\cdot \phi(y)$,
\newline
$(2)$ $\phi(y\cdot d)=\phi(y)\cdot d$,
\newline
$(3)$ ${}_{C_1} \la \phi(y), \, \phi(z) \ra ={}_C \la y, z \ra$,
\newline
$(4)$ $\la \phi(y), \, \phi(z) \ra_{D_1} =\la y, z \ra_D$.
\end{lemma}
\begin{proof}
Let $c\in C$, $d\in D$, $y, z\in Y$. Then
\begin{align*}
\phi(c\cdot y) & =\sum_{i, j}u_i \otimes E^X (u_i^* c\cdot y\cdot v_j )\otimes \widetilde{v_j }
=\sum_{i, j, k}u_i \otimes E^X (E^A (u_i^* cu_k )u_k^* \cdot y \cdot v_j )\otimes \widetilde{v_j } \\
& =\sum_{i, j, k}u_i E^A (u_i^* cu_k )\otimes E^X (u_k^* \cdot y\cdot v_j )\otimes\widetilde{v_j }
=\sum_{j, k}cu_k \otimes E^X (u_k^* \cdot y\cdot v_j )\otimes\widetilde{v_j} \\
& =c\cdot \phi(y) .
\end{align*}
Hence we obtain Condition (1). In the similar way to the above, we can obtain Condition (2).
Next we show Conditions (3) and (4).
\begin{align*}
{}_{C_1} \la \phi(y), \, \phi(z) \ra & =\sum_{i, j, k, l}{}_{C_1} 
\la u_i \otimes E^X (u_i^* \cdot y\cdot v_j )\otimes \widetilde{v_j}, \,
u_k \otimes E^X (u_k^* \cdot z\cdot v_l )\otimes\widetilde{v_l} \ra \\
& =\sum_{i, j, k, l}{}_{C_1} \la u_i {}_A \la E^X (u_i^* \cdot y\cdot v_j )\otimes\widetilde{v_j}, \,
E^X (u_k^* \cdot z\cdot v_l )\otimes\widetilde{v_l} \ra , \, u_k \ra \\
& =\sum_{i, j, k, l}u_i {}_A \la E^X (u_i^* \cdot y\cdot v_j )\otimes \widetilde{v_j}, \, 
E^X (u_k^* \cdot z\cdot v_l )\otimes\widetilde{v_l} \ra e_A u_k^* \\
& =\sum_{i, j, k, l}u_i {}_A \la E^X (u_i^* \cdot y\cdot v_j )\cdot \la v_j , \, v_l \ra_B , \,
E^X (u_k^* \cdot z\cdot v_l ) \ra e_A u_k^* \\
& =\sum_{i, j, k, l}u_i {}_A \la E^X (u_i^* \cdot y\cdot v_j )\cdot E^B (v_j^* v_l ), \, E^X (u_k^* \cdot z\cdot v_l ) \ra 
e_A u_k^* \\
& =\sum_{i, j, k, l}u_i {}_A \la E^X (u_i^* \cdot y\cdot v_j E^B (v_j^* v_l )), \, E^X (u_k^* \cdot z\cdot v_l ) \ra e_A u_k^* \\
& =\sum_{i, k, l}u_i {}_A \la E^X (u_i^* \cdot y\cdot v_l ), \, E^X (u_k^* \cdot z\cdot v_l ) \ra e_A u_k^* \\
& =\sum_{i, k, l}u_i E^A ({}_C \la u_i^* \cdot y\cdot v_l , \, E^X (u_k^* \cdot z \cdot v_l ) \ra )e_A u_k^* \\
& =\sum_{i, k, l}u_i  E^A (u_i^* {}_C \la y\cdot v_l , \, E^X (u_k^* \cdot z \cdot v_l ) \ra )e_A u_k^* \\
& =\sum_{k, l} {}_C \la y\cdot v_l , \, E^X (u_k^* \cdot z\cdot v_l  ) \ra e_A u_k^* \\
& =\sum_{k, l}{}_C \la y, \, E^X (u_k^* \cdot z\cdot v_l )\cdot v_l^* \ra e_A u_k^* \\
&=\sum_k {}_C \la y, u_k^* \cdot z \ra e_A u_k^* \\
& =\sum_k {}_C \la y, z \ra u_k e_A u_k^* \\
& ={}_C \la y, z \ra .
\end{align*}
Hence we obtain Condition (3). Similarly we obtain Condition (4).
\end{proof}

By the above lemma, we can identify $Y$ with a closed subspace of
$Y_1$ satisfying Conditions (1), (2) in Definition \ref{def:inclusion} except the
conditions that ${}_C \la Y_1 , Y \ra=C$ and $\la Y_1 , Y \ra_D =D$.

\begin{lemma}\label{lem:condition2}With the above, we identify $Y$ with a closed subspace
of $Y_1$ by the linear map $\phi$. Then ${}_{C_1} \la Y_1 , Y \ra=C_1$ and $\la Y_1 , Y \ra_{D_1} =D_1$.
\end{lemma}
\begin{proof}Let $c\otimes x\otimes\widetilde{d}\in Y_1$ and $y\in Y$. Since
$\phi(y)=\sum_{i, j}u_i \otimes E^X (u_i^* \cdot y\cdot v_j )\otimes\widetilde{v_j}$,
\begin{align*}
{}_{C_1 }\la c\otimes x\otimes\widetilde{d}, \phi(y) \ra & =\sum_{i, j}{}_{C_1} \la c\otimes x\otimes \widetilde{d} , \, u_i
\otimes E^X (u_i^* \cdot y\cdot v_j )\otimes\widetilde{v_j} \ra \\
& =\sum_{i, j}{}_{C_1} \la c\cdot {}_A \la x\otimes\widetilde{d} \, ,\, E^X
(u_i^* \cdot y\cdot v_j )\otimes\widetilde{v_j } \ra , \, u_i \ra \\
& =\sum_{i, j}{}_{C_1} \la c\cdot {}_A \la x\cdot E^B (d^* v_j ), \, E^X (u_i^* \cdot y\cdot v_j ) \ra , \, u_i \ra \\
& =\sum_{i, j}c{}_A \la x\cdot E^B (d^* v_j ), \, E^X (u_i^* \cdot y\cdot v_j ) \ra e_A u_i^* \\
& =\sum_{i, j}ce_A \, {}_A \la x\cdot E^B (d^* v_j ), \, E^X (u_i^* \cdot y\cdot v_j ) \ra u_i^* \\
& =\sum_{i, j}ce_A \, {}_C \la x \cdot E^B (d^* v_j ), \, u_i \cdot E(u_i^* \cdot y\cdot v_j ) \ra \\
& =\sum_j ce_A \, {}_C \la x\cdot E^B (d^* v_j ), \, y\cdot v_j \ra \\
& =\sum_j ce_A \, {}_C \la x\cdot E^B (d^* v_j )v_j^* ,, y \ra \\
& =ce_A \, {}_C \la x\cdot d^* , \, y \ra =ce_A \, {}_C \la x, y\cdot d \ra .
\end{align*}
Since ${}_C \la X, Y\ra =C$, we obtain that ${}_{C_1} \la Y_1 , Y \ra=C_1$. Also,
since $\la X, Y \ra_D =D$, we obtain that $\la Y_1 , Y \ra_{D_1}=D_1$ in the same way as above.
\end{proof}

By Lemmas \ref{lem:bimodule} and \ref{lem:condition2}, we obtain the following corollary:

\begin{cor}\label{cor:dual}With the above notations, the inclusions $C\subset C_1$
and $D\subset D_1$ are strongly Morita equivalent with respect to
the $C_1 -D_1$-equivalence bimodule $Y_1$ and its closed subspace
$Y$.
\end{cor}

Let $E^C$ and $E^D$ be the dual conditional expectations of $E^A$ and
$E^B$, respectively.

\begin{lemma}\label{lem:EY}With the above notations, $E^Y$ is a conditional
expectation from $Y_1$ onto $Y$ with respect to $E^C$ and $E^D$.
\end{lemma}
\begin{proof}
We show that Conditions (1)-(6) in Definition \ref{def:expectation} hold.
We note that we identify $Y$ with $\phi(Y)\subset Y_1$.
\newline
(1) For any $c_1, c_2\in C$, $y\in Y$,
\begin{align*}
E^Y (c_1 e_A c_2 \cdot y) & =\sum_{i, j}E^Y (c_1 e_A c_2 \cdot u_i \otimes E^X (u_i^* \cdot y\cdot v_j )\otimes \widetilde{v_j}) \\
& =\sum_{i, j}E^Y (c_1 E^A (c_2 u_i )\otimes E^X (u_i^* \cdot y\cdot v_j )\otimes\widetilde{v_j}) \\
& =\sum_{i, j}\Ind_W (E^A )^{-1}c_1 E^A (c_2 u_i )\cdot E^X (u_i^* \cdot y\cdot v_j )\cdot v_j^* \\
& =\Ind_W (E^A )^{-1}c_1 c_2 \cdot y =E^C (c_1 e_A c_2 )\cdot y .
\end{align*}
(2) For any $c_1 , c_2 \in C$, $x\in X$, $d\in D$,
\begin{align*}
E^Y (c_1 \cdot c_2 \otimes x\otimes \widetilde{d}) & =E^Y (c_1 c_2 \otimes x\otimes\widetilde{d})
=\Ind_W (E^A )^{-1}c_1 c_2 \cdot x\cdot d^* \\
& =c_1 \cdot E^Y (c_2 \otimes x\otimes \widetilde{d} ) .
\end{align*}
(3) By the proof of Lemma \ref{lem:condition2}, for any $c\in C$, $d\in D$, $x\in X$, $y\in Y$,
\begin{align*}
E^C ({}_{C_1} \la c\otimes x\otimes \widetilde{d}, \, y \ra ) 
& =\Ind_W (E^A )^{-1} c\, \, {}_C \la x\cdot d^* , \, y \ra \\
& =\Ind_W (E^A )^{-1}{}_C \la c\cdot x\cdot d^* , \, y \ra
= {}_{C_1}\la E^Y (c\otimes x\otimes\widetilde{d}), \, y \ra .
\end{align*}
(4) By Lemma \ref{lem:index}, we can see that
$$
E^Y (y\cdot d_1 e_B d_2 )=y\cdot E^D (d_1 e_B d_2 )
$$
for any $d_1 , d_2 \in D$, $y\in Y$ in the same way as in the proof of
Condition (1).
\newline
(5) In the same way as in the proof of Condition (2), we can see that
$$
E^Y (c\otimes x\otimes \widetilde{d_1}\cdot d_2 )=E^Y (c\otimes x\otimes \widetilde{d_1})\cdot d_2
$$
for any $c\in C$, $d_1 , d_2 \in D$, $x\in X$.
\newline
(6) By Lemma \ref{lem:index} we can see that
$$
E^B (\la c\otimes x\otimes \widetilde{d}, y \ra_{D_1}= \la E^Y c\otimes x\otimes \widetilde{d}), \, y \ra_{D_1} .
$$
for any $c\in C$, $d\in D$, $x\in X$, $y\in Y$.
Therefore we obtain the conclusion.
\end{proof}

\begin{Def}\label{def:dual}In the above situation, $Y_1$ is called the
\sl
upward basic construction
\rm
of $Y$ for $E^X$. Also, $E^Y$ is called the
\sl
dual
\rm
conditional expectation of $E^X$.
\end{Def}

\begin{remark}\label{remark:another}The linear map $\phi$ from $Y$ to $Y_1$ defined in the above
is independent of the choice of quasi-bases $\{(u_i , u_i^* )\}$ and $\{(v_j , v_j^* )\}$ for
$E^A$ and $E^B$, respectively. Indeed, let $\{(w_i , w_i^* )\}$ and $\{(z_j , z_j^* )\}$ be another pair of
quasi-bases for $E^A$ and $E^B$, respectively. Then for any $y\in Y$,
\begin{align*}
\sum_{i, j}w_i \otimes E^X (w_i^* \cdot y\cdot z_j )\otimes\widetilde{z_j} & =\sum_{i, j, k, l}
u_k E^A (u_k^* w_i )\otimes E^X (w_i^* \cdot y\cdot z_j )\otimes [v_l E^B (v_l^* z_j )]^{\widetilde{}} \\
& =\sum_{i, j, k, l}u_k \otimes E^X (E^A (u_k^* w_i )w_i^* \cdot y\cdot z_j )\otimes E^B (z_j^* v_l )\cdot\widetilde{v_l} \\
& =\sum_{j, k, l}u_k \otimes E^X (u_k^* \cdot y\cdot z_j E^B (z_j^* v_l ))\otimes\widetilde{v_l} \\
& =\sum_{k. l}u_k \otimes E^X (u_k^* \cdot y\cdot v_l )\otimes\widetilde{v_l} =\phi(y) .
\end{align*}
\end{remark}

Next, we shall show that the upward basic construction for equivalence bimodules is unique
in a certain sense.

Let $A\subset C$ and $B\subset D$ be unital inclusions of unital
$C^*$-algebras as above. Also, let $E^A , E^B , E^X$ and $C_1, D_1$ be as above.

\begin{lemma}\label{lem:index2}With the above notations, $\Ind_W (E^A )\in A$ if
and only if $\Ind_W (E^B )\in B$.
\end{lemma}
\begin{proof}
We assume that $\Ind_W (E^A )\in A$. By the discussions before Lemma \ref{lem:expectation},
we may assume that
$$
B=pM_k (A)p, \quad D=pM_k (C)p, \quad E^B (E^A \otimes\id_{M_k (\BC)})|_{pM_k (C)p},
$$
where $k\in \BN$ and $p$ is a projection in $M_k (A)$ satisfying that
$M_k (A)pM_k (A)=M_k (A)$ and $M_k (C)pM_k (C)=M_k (C)$.
Then by the discussions before Lemma \ref{lem:expectation},
$$
\Ind_W (E^B )=(\Ind_W (E^A )\otimes I_k )p .
$$
Since $\Ind_W (E^A )\in A$, $\Ind_W (E^B )\in pM_k (A)p=B$.
Thus, we obtain the conclusion.
\end{proof}

Let $W$ be a $C_1 - D_1$-equivalence bimodule.
We suppose that $\Ind_W (E^A )\in A$. Then
$\Ind_W (E^B )\in B$ by Lemma \ref{lem:index2}. Also,
we suppose that $Y$ is included in $W$ as its closed subspace
and that the inclusions $C\subset C_1$ and $D\subset D_1$ are strongly
Morita equivalent with respect to $W$ and its closed subspace $Y$.
Furthermore, we suppose that there is a conditional expectation $F^Y$ from
$W$ onto $Y$ with respect to $E^C$ and $E^D$ satisfying that
$$
F^Y (e_A \cdot y\cdot e_B )=\Ind_W (E^A )^{-1}\cdot E^X (y) \qquad (*)
$$
for any $y\in Y$, where $e_A$ and $e_B$ are the Jones projections for $E^A$ and
$E^B$, respectively. We note that in Lemma \ref{lem:commute}, we shall show that the conditional
expectation $E^Y$ from $Y_1$ onto $Y$ with respect to $E^C$ and $E^D$ satisfies that
$$
E^Y (e_A \cdot y \cdot e_B )=\Ind_W (E^A )^{-1}\cdot E^X (y)
$$
for any $y\in Y$. We show that there is a $C_1 -D_1$-equivalence bimodule isomorphism $\theta$
from $W$ onto $Y_1$ such that
$$
F^Y =E^Y \circ \theta .
$$
Let $\{(u_i , u_i^* )\}_{i=1}^n $ and $\{(v_j , v_j^* )\}_{j=1}^m$ be quasi-bases for $E^A$ and $E^B$,
respectively and let $\{(w_i , w_i^* )\}_{i=1}^n$ and $\{(z_j , z_j^* )\}_{j=1}^m$ be their dual
quasi-bases for $E^C$ and $E^D$ defined by
\begin{align*}
w_i & =u_i e_A \Ind_W (E^A )^{\frac{1}{2}} \, , (i=1,2,\dots, n), \\
z_j & =v_j e_B \Ind_W (E^B )^{\frac{1}{2}} \, , (j=1,2\dots, m) ,
\end{align*}
respectively. Let $\theta$ be the map from $W$ to $Y_1$ defined by
\begin{align*}
\theta(y) & =\Ind_W (E^A )\sum_{i, j}u_i \otimes E^X (F^Y (e_A u_i^* \cdot y\cdot v_j e_B ))\otimes \widetilde{v_j} \\
& =\sum_{i, j}u_i \otimes E^X (F^Y (e_A u_i^* \cdot y \cdot v_j e_B ))\otimes\widetilde{v_j}\cdot \Ind_W (E^B ) .
\end{align*}
for any $y\in W$. Clearly $\theta$ is a linear map from $W$
to $Y_1$.

\begin{lemma}\label{lem:bimodule2}With the above notations, for any $c_1 , c_2 \in C$,
$d_1 , d_2 \in D$ and $y\in W$,
$$
\theta(c_1 e_A c_2 \cdot y) =c_1 e_A c_2 \cdot \theta(y), \quad
\theta(y\cdot d_1 e_B d_2 ) =\theta(y)\cdot d_1 e_B d_2 .
$$
\end{lemma}
\begin{proof}For any $c_1 , c_2 \in C$ and $y\in W$,
\begin{align*}
\theta (c_1 e_A c_2 \cdot y) & =\Ind_W (E^A )\sum_{i, j}u_i \otimes E^X (F^Y (E^A (u_i^* c_1 )e_A c_2 \cdot y
\cdot v_j e_B ))\otimes \widetilde{v_j } \\
& =\Ind_W (E^A )\sum_{i, j}u_i E^A (u_i^* c_1 )\otimes E^X (F^Y (e_A c_2 \cdot y\cdot v_j e_B ))\otimes\widetilde{v_j} \\
& =\Ind_W (E^A )\sum_{i, j}c_1 \otimes E^X (F^Y (e_A E^A (c_2 u_i )u_i^* \cdot y\cdot v_j e_B ))\otimes\widetilde{v_j} \\
& =\Ind_W (E^A )\sum_{i, j}c_1 e_A c_2 \cdot u_i \otimes E^X (F^Y (e_A u_i^* \cdot y\cdot v_j e_B ))\otimes\widetilde{v_j} \\
& =c_1 e_A c_2 \cdot \theta (y) .
\end{align*}
Similarly we can see that $\theta(y\cdot d_1 e_B d_2 )=\theta(y)\cdot d_1 e_B d_2 $ for any
$d_1 , d_2 \in D$ and $y\in W$. Therefore, we obtain the conclusion.
\end{proof}

\begin{lemma}\label{lem:surjective2}With the above notations, $\theta$ is surjective.
\end{lemma}
\begin{proof}By Lemma \ref{lem:bimodule2} and Condition ($*$), for any $c\in C$, $d\in D$ and $x\in X$
\begin{align*}
\theta(ce_A \cdot x\cdot e_B d^* ) & =ce_A \cdot \theta(x)\cdot e_B d^* \\
& =\sum_{i, j}ce_A \cdot u_i \otimes E^X (u_i^* \cdot x\cdot v_j )\otimes \widetilde{v_j}\cdot e_B d^* \\
& =\sum_{i, j}c\otimes E^X (E^A (u_i )u_i^* \cdot x\cdot v_j E^B (v_j^* ))\otimes \widetilde{d}
=c\otimes x\otimes\widetilde{d} .
\end{align*}
Hence $\theta$ is surjective.
\end{proof}

Next, we show that $\theta$ preserves the both-sided inner products.

\begin{lemma}\label{lem:commute}For any $y\in Y$,
\begin{align*}
& e_A \cdot y\cdot e_B =e_A \cdot \phi(y)\cdot e_B =e_A \cdot E^X (y)=E^X (y)\cdot e_B , \\
& E^Y (e_A \cdot y\cdot e_B ) =\Ind_W (A)^{-1}\cdot E^X (y)=E^X (y)\cdot \Ind_W (B)^{-1} .
\end{align*}
\end{lemma}
\begin{proof}For any $y\in Y$,
\begin{align*}
e_A \cdot y\cdot e_B & =e_A \cdot \sum_{i, j}u_i \otimes E^X (u_i^* \cdot y\cdot v_j )\otimes\widetilde{v_j}\cdot e_B \\
& =\sum_{i, j}1\otimes E^X (E^A (u_i )u_i^* \cdot y\cdot v_j E^B (v_j^* ))\otimes\widetilde{1}
=1\otimes E^X (y)\otimes \widetilde{1} .
\end{align*}
Also, by the similar computations to the above, for any $y\in Y$
$$
e_A \cdot E^X (y)=e_A \cdot \phi(E^X (y))=E^X (y)\cdot e_B =1\otimes E^X (y)\otimes\widetilde{1} .
$$
Furthermore,
\begin{align*}
E^Y (e_A \cdot y\cdot e_B ) & =E^Y(e_A \cdot E^X (y))=E^C (e_A )\cdot E^X (y) \\
& =\Ind_W (A)^{-1}\cdot E^X (y)=E^X (y)\cdot \Ind_W (B)^{-1}
\end{align*}
by Lemma \ref{lem:index}.
Thus, we obtain the conclusion.
\end{proof}
\begin{lemma}\label{lem:preserving}With the above notations,
$\theta$ preserves the both-sided inner products.
\end{lemma}
\begin{proof}Let $y_1 , y_2 \in W$. Then
$$
\theta(y_1 ) =\Ind_W (E^A ) \sum_{i, j}u_i \otimes x_1 \otimes\widetilde{v_j }, \quad
\theta(y_2 ) =\Ind_W (E^A ) \sum_{i_1 , j_1 }u_{i_1} \otimes x_2 \otimes\widetilde{v_{j_1} } ,
$$
where 
$$
x_1 =E^X (F^Y (e_A u_i^* \cdot y_1 \cdot v_j e_B )) , \quad x_2 =
E^X (F^Y (e_A u_{i_1}^* \cdot y_2 \cdot v_{j _1}e_B )) .
$$
Hence by Lemma \ref{lem:commute},
\begin{align*}
& {}_{C_1} \la \theta(y_1 ), \theta(y_2) \ra =\Ind_W (E^A )^2 \sum_{i, j, i_1 , j_1 }
{}_{C_1} \la u_i \otimes x_1 \otimes\widetilde{v_j} , u_{i_1}\otimes x_2 \otimes\widetilde{v_{j_1}} \ra \\
& =\Ind_W (E^A )^2 \sum_{i, j, i_1 , j_1}{}_{C_1}\la u_i \, {}_A \la x_1 \otimes \widetilde{v_j},
x_2 \otimes \widetilde{v_{j_1}} \ra , u_{i_1} \ra \\
& =\Ind_W (E^A )^2\sum_{i, j, i_1 , j_1}{}_{C_1} \la u_i \, {}_A \la x_1 \cdot  {}_B \la \widetilde{v_j}, 
\widetilde{v_{j_1}} \ra , x_2 \ra , u_{i_1} \ra \\
& =\Ind_W (E^A )^2\sum_{i, j, i_1 , j_1}{}_{C_1} \la u_i \, {}_A \la x_1 \cdot \la v_j, 
v_{j_1} \ra_B , x_2 \ra , u_{i_1} \ra \\
& =\Ind_W (E^A )^2\sum_{i, j, i_1 , j_1}{}_{C_1} \la u_i \, {}_A \la x_1 \cdot E^B (v_j^* v_{j_1}) ,
x_2 \ra , u_{i_1} \ra \\
& =\Ind_W (E^A )^2 \sum_{i, j, i_1 , j_1}u_i e_A \, {}_A \la x_1 \cdot E^B (v_j^* v_{j_1}), x_2 \ra u_{i_1}^* \\
& =\Ind_W (E^A )^2 \\
& \times \sum_{i, i_1 ,  j_1}u_i e_A \, {}_A \la E^X(F^Y (e_A u_i^* \cdot y_1 \cdot v_{j_1} e_B)) , \,
E^X (F^Y (e_A u_{i_1}^* \cdot y_2 \cdot v_{j_1}e_B )) \ra u_{i_1}^* \\
& =\Ind_W (E^A )^2 \\
& \times \sum_{i, i_1 , j_1}u_i \, {}_{C_1} \la e_A \cdot F^Y (e_A u_i^* \cdot y_1 \cdot v_{j_1}e_B)\cdot e_B , \,
e_A \cdot F^Y (e_A u_{i_1}^* \cdot y_2 \cdot v_{j_1}e_B )\cdot e_B \ra u_{i_1}^* \\
& =\Ind_W (E^A )^2 \\
& \times \sum_{i, i_1 , j_1}{}_{C_1} \la u_i e_A \cdot F^Y (e_A u_i^* \cdot y_1 \cdot v_{j_1}e_B )\cdot e_B , \,
u_{i_1}e_A \cdot F^Y (e_A u_{i_1}^* \cdot y_2 \cdot v_{j_1}e_B )\cdot e_B \ra \\
& =\sum_{i, i_1 , j_1} {}_{C_1} \la w_i \cdot F^Y (w_i^* \cdot y_1 \cdot v_{j_1}e_B )\cdot e_B \,, \,
w_{i_1}\cdot F^Y (w_{i_1}^* \cdot y_2 \cdot v_{j_1}e_B )\cdot e_B \ra \\
& =\sum_{j_1}{}_{C_1} \la y_1 \cdot v_{j_1}e_B , \,  y_2 \cdot v_{j_1}e_B \ra 
=\sum_{j_1}{}_{C_1} \la y_1 \cdot v_{j_1}e_B v_{j_1}^* , \, y_2 \ra
={}_{C_1} \la y_1 , y_2 \ra .
\end{align*}
Also, by Lemma \ref{lem:commute}, we ca see that
$\la \theta(y_1 ), \theta(y_2) \ra_{D_1}=\la y_ 1 , y_2 \ra_{D_1}$
in the same way as in the above. Therefore, we obtain the conclusion.
\end{proof}

\begin{prop}\label{prop:universal}With the above notations, $\theta$ is a
$C_1 -D_1$-equivalence bimodule isomorphism from $W$ onto $Y_1$ such that $F^Y =E^Y \circ \theta$.
\end{prop}
\begin{proof}By Lemmas \ref{lem:bimodule2}, \ref{lem:surjective2} and \ref{lem:preserving},
we have only to show that $F^Y =E^Y \circ\theta$.
For any $y\in W$,
\begin{align*}
(E^Y \circ\theta)(y) & =\sum_{i, j}u_i\cdot E^X (F^Y (e_A u_i^* \cdot y\cdot v_j e_B ))\cdot v_j^* \\
& =\Ind_W (E^A )\sum_{i, j}u_i \cdot F^Y (e_A \cdot F^Y (e_A u_i^* \cdot y\cdot v_j e_B )\cdot e_B )\cdot v_j^* \\
& =\Ind_W (E^A )\sum_{i, j}F^Y (u_i e_A \cdot F^Y (e_A u_i^* \cdot y\cdot v_j e_B )\cdot e_B v_j^* ) \\
& =\Ind_W (E^A )^{-1}\sum_{i, j}F^Y (w_i \cdot F^Y (w_i^* \cdot y\cdot z_j )\cdot z_j^* ) \\
& =\Ind_W (E^A )^{-1}\sum_j F^Y (y\cdot z_j z_j^* ) \\
& =F^Y (y)
\end{align*}
by Condition ($*$) and Lemma \ref{lem:index}. Therefore, we obtain the conclusion.
\end{proof}

\section{Duality}\label{sec:duality}
In this section, we shall present a certain duality theorem for inclusions of
equivalence bimodules.
\par
Let $A\subset C$ and $B\subset D$ be unital inclusions of unital $C^*$-algebras,
which are strongly Morita equivalent with respect to a $C-D$-equivalence bimodule $Y$
and its closed subspace $X$. Let $E^A$ and $E^B$ be conditional
expectations of Watatani index-finite type from $C$ and $D$ onto $A$ and $B$, respectively.
Let $E^X$ be a conditional expectation from $Y$ onto $X$ with respect to
$E^A$ and $E^B$. Let $C_1$ and $D_1$ be the
$C^*$-basic constructions for $E^A$ and $E^B$ and $e_A$ and $e_B$
the Jones projections for $E^A$ and $E^B$, respectively. Let $Y_1$ be the upward
basic construction for $E^X$ and let $E^C$, $E^D$ and $E^Y$ be the dual conditional
expectations from $C_1$, $D_1$ and $Y_1$ onto $C$, $D$ and $Y$, respectively.
Furthermore, let $C_2$ and $D_2$ be the $C^*$-basic constructions for
$E^C$ and $E^D$, respectively and $e_C$ and $e_D$ the Jones projections for
$E^C$ and $E^D$, respectively. Let $Y_2$ be the upward basic construction for $E^Y$ and
let $E^{C_1}$, $E^{D_1}$ and $E^{Y_1}$ be the
dual conditional expectations from $C_2$, $D_2$ and $Y_2$ onto $C_1$, $D_1$ and $Y_1$,
respectively. Let $\{(u_i , u_i^* )\}_{i=1}^k$ and $\{(v_i , v_i^* )\}_{i=1}^{k_1}$ be
quasi-bases for $E^A$ and $E^B$, respectively. We note that we can assume that
$k=k_1$.

We suppose that $\Ind_W (E^A )\in A$. Then
$\Ind_W (E^B )\in B$ by Lemma \ref{lem:index}. By Proposition \ref{prop:morita3}, the
inclusions $C_1 \subset C_2$ and $A\subset C$ are strongly Morita equivalent
with respect to the $C_2 -C$-equivalence bimodule $C_1$ and its closed subspace $C$.
Also, there is a conditional expectation $G$ from $C_1$ onto $C$ with respect to $E^C$ and $E^A$.
Let $p=[E^A (u_i^* u_j )]_{i, j=1}^k$. Then by the discussions in Section \ref{sec:definition},
$p$ is a full projection in $M_k (A)$.
Let $\Psi_{C_1}$ be the map from $C_1$ to $M_k (A)$ defined by
$$
\Psi_{C_1} (c_1 e_A c_1 )=[E^A (u_i^* c_1 )E^A (c_2 u_j )]_{i, j=1}^k
$$
for any $c_1 , c_2 \in C$. Then by the discussions in Section \ref{sec:definition},
$\Psi_{C_1}$ is an isomorphism of $C_1$ onto $pM_k (A)p$. Let $\Psi_{C_2}$ be the map
from $C_2$ to $M_k (C)$ defined by
\begin{align*}
\Psi_{C_2} (c_1 e_C c_2 ) & =[E^C (w_i^* c_1 )E^C (c_2 w_j )]_{i, j=1}^k \\
& =[E^C (\Ind_W (E^A )^{\frac{1}{2}}e_A u_i^* c_1 )E^C (\Ind_W (E^A )^{\frac{1}{2}}c_2 u_j e_A )] \\
& =[\Ind_W (E^A )E^C (e_A u_i^* c_1 )E^C (c_2 u_j e_A )]
\end{align*}
for any $c_1 , c_2 \in C_1$, where $\{(w_i, w_i^* )\}_{i=1}^k$ is the quasi-basis for $E^C$ defined by
$w_i =\Ind_W (E^A )^{\frac{1}{2}}u_i e_A$ for $i=1,2,\dots, k$.
Then $\Psi_{C_2}$ is also an isomorphism of $C_2$ onto $pM_k (C)p$.
Furthermore, let $\Phi_C$ be the map from $C$ to $M_k (A)$ defined by
$$
\Phi_C (c)=\left[
\begin{array}{ccc}
E^A (u_1^* c) \\
\vdots \\
E^A (u_k^* c)
\end{array} \right]
$$
for any $c\in C$, By the discussions in Section \ref{sec:definition}, $\Phi_C$ is a $C_1 - A$-equivalence bimodule isomorphism
of the $C_1 -A$-equivalence bimodule $C$ onto the $pM_k (A)p -A$-equivalence bimodule $pM_k (A)(1\otimes f)$,
where $f=\begin{bmatrix} 1 & 0 & \ldots & 0 \\
0 & 0 & \ldots & 0 \\
\vdots & \vdots & \ddots & \vdots \\
0 & 0 & \ldots & 0 \end{bmatrix}\in M_k (\BC)$ and we identify $A$
and $C_1$ with $A\otimes f$ and $pM_k (A)p$, respectively.
Let $\Phi_{C_1}$ be the map from $C_1$ to $M_k (C)$ defined by
\begin{align*}
\Phi_{C_1}(c) & =\left[
\begin{array}{ccc}
E^C (w_1^* c ) \\
\vdots \\
E^C (w_k^* c)
\end{array} \right]
\end{align*}
for any $c\in C$. Then by the discussions in Section \ref{sec:definition},
$\Phi_{C_1}$ is a $C_2 -C$-equivalence bimodule isomorphism of the $C_2 -C$-equivalence
bimodule $C_1$ onto the $pM_k (C)p-C$-equivalence bimodule
$pM_k (C)(1\otimes f)$, where $f=\begin{bmatrix} 1 & 0 & \ldots & 0 \\
0 & 0 & \ldots & 0 \\
\vdots & \vdots & \ddots & \vdots \\
0 & 0 & \ldots & 0 \end{bmatrix}\in M_k (\BC)$ and
we identify $C$ and $C_2$ with $C\otimes f$ and $pM_k (C)p$, respectively.
Thus, the inclusion $C_1 \subset C_2$ can be identified with the inclusion
$pM_k (A)p \subset pM_k (C)p$ , the $C_1 -A$-equivalence bimodule $C$ can be identified with
the $pM_k (A)p-A$-equivalence bimodule $pM_k (A)(1\otimes f)$ and $E^C$ can be
identified with $(E^A \otimes\id)|_{pM_k (A)p}$ by the above isomorphisms.
Similar results to the above hold, that is,
let $q=[E^B (v_i^* v_j )]_{i, j=1}^k$. Then $q$ is a full projection in $M_k (B)$
Then the inclusion $D_1 \subset D_2$
is identified the inclusion $qM_k (B)q\subset qM_k (D)q$, the $D_1 -B$-equivalence bimodule $D$
is identified with $qM_k (B)q-B$-equivalence bimodule $qM_k (B)(1\otimes f)$ and $E^D$ is
identified with $(E^D \otimes\id)|_{qM_k (B)q}$ by the following isomorphisms: Let $\Psi_{D_1}$ be the isomorphism
of $D_1$ onto $qM_k (B)q$ defined by
$$
\Psi_{D_1}(d_1 e_B d_2)=[E^B (v_i^* d_1 )E^B (d_2 v_j )]_{i, j=1}^k ,
$$
for any $d_1, d_2 \in D$. Let $\Psi_{D_2}$ be the isomorphism of $D_2$ onto $qM_k (D)q$
defined by
$$
\Psi_{D_2}(d_1 e_D d_2)=[E^D (z_i^* d_1 )E^D (d_2 z_j )]_{i, j=1}^k
$$
for any $d_1 , d_2 \in D_1$, where $\{(z_i , z_i^*)\}_{i=1}^k$ is the quasi-basis
for $E^D$ defined by $z_i =\Ind_W (B)^{\frac{1}{2}}v_i e_B$ for $i=1,2,\dots, k$. Furthermore,
let $\Phi_D$ be the $D_1 - B$-equivalence bimodule isomorphism of $D$ onto $qM_k (B)(1\otimes f)$ defined by
$$
\Phi_D (d)=\left[
\begin{array}{ccc}
E^B (v_1^* d) \\
\vdots \\
E^B (v_k^* d)
\end{array} \right]
$$
for any $d\in D$, where we identify $D_1$ with $qM_k (B)q$.
Let $\Phi_{D_1}$ be the $D_2 -D$-equivalence bimodule isomorphism of $D_1$ onto $qM_k (D)(1\otimes f)$
defined by
$$
\Phi_{D_1} (d)=\left[
\begin{array}{ccc}
E^D (z_1^* d) \\
\vdots \\
E^D (z_k^* d)
\end{array} \right]
$$
for any $d\in D_1$, where we identify $D_2$ with $qM_k (D)q$.

Let $Y_1$ and $Y_2$ be the upward basic constructions for $E^X$ and $E^Y$, respectively.
By the definitions of $Y_1$ and $Y_2$,
$$
Y_1 =C\otimes_A X \otimes_B \widetilde{D}, \quad Y_2 =C_1 \otimes_C Y\otimes_D \widetilde{D_1} .
$$
Then
$$
Y_1 \cong pM_k (A)(1\otimes f)\otimes_A X \otimes_B (1\otimes f)M_k (B)q
$$
as $C_1 - D_1$-equivalence bimodules
where we identify $pM_k (A)p$ and $qM_k (B)q$ are identified with
$C_1$ and $D_1$, respectively. We regard $p\cdot M_k (X)\cdot q$ as a
$pM_k (A)p-qM_k (B)q$-equivalence bimodule in the usual way.
Similarly
$$
Y_2 \cong pM_k (C)(1\otimes f)\otimes_C Y\otimes_D (1\otimes f)M_k (D)q
$$
as $C_2 -D_2$-equivalence bimodules, where we identify $pM_k (C)p$ and $qM_k (D)q$ are identified with
$C_2$ and $D_2$, respectively. 

\begin{lemma}\label{lem:bimodule3}With the above notations,
$$
pM_k (A)(1\otimes f)\otimes_A X \otimes_B (1\otimes f)M_k (B)q\cong p\cdot M_k (X)\cdot q
$$
as $pM_k (A)p-qM_k (B)q$-equivalence bimodules. Hence $Y_1 \cong p\cdot M_k (X)\cdot q$
as $C_1 -D_1 $-equivalence bimodules, where we identify $pM_k (A)p$ and $qM_k (B)q$
with $C_1$ and $D_1$, respectively.
\end{lemma}
\begin{proof}
We have only to show that
$$
pM_k (A)(1\otimes f)\otimes_A X \otimes_B (1\otimes f)M_k (B)q\cong p\cdot M_k (X)\cdot q
$$
as $pM_k (A)p-qM_k (B)q$-equivalence bimodules. Let $\Phi$ be the map from
$pM_k (A)(1\otimes f)\otimes_A X \otimes_B (1\otimes f)M_k (B)q$ to $p\cdot M_k (X)\cdot q$
defined by
$$
\Phi(pa(1\otimes f)\otimes x\otimes (1\otimes f)bq)=pa\cdot (x\otimes f)\cdot bq
$$
for any $a\in M_k (A)$, $b\in M_k (B)$, $x\in X$. Then it is clear that $\Phi$
is well-defined and a $pM_k (A)p-qM_k (B)q$-bimodule. For any $a_1 ,a_2 \in M_k (A)$,
$b_1 , b_2 \in M_k (B)$ and $x_1, x_2 \in X$,
\begin{align*}
& {}_{pM_k (A)p} \la pa_1 (1\otimes f)\otimes x_1 \otimes (1\otimes f)b_1 q , \,
pa_2 (1\otimes f)\otimes x_2 \otimes (1\otimes f)b_2 q \ra \\
& ={}_{pM_k (A)p} \la pa_1 (1\otimes f)\cdot {}_A \la x_1 \otimes (1\otimes f)b_1 q, \,
x_2 \otimes (1\otimes f)b_2 q \ra, \, pa_2 (1\otimes f) \ra \\
& ={}_{pM_k (A)p} \la pa_1 {}_A \la x_1 \otimes (1\otimes f)b_1 q, \,
x_2 \otimes (1\otimes f)b_2 q \ra \otimes f, \, pa_2 (1\otimes f) \ra \\
& =pa_1 [{}_A \la x_1 \otimes (1\otimes f)b_1 q, \, x_2 \otimes (1\otimes f)b_2 q \ra\otimes f]a_2^* p \\
& =pa_1 [{}_A \la x_1 \cdot {}_B \la (1\otimes f)b_1 q, \, (1\otimes f)b_2 q \ra, \, x_2 \ra \otimes f]a_2^* p \\
& =pa_1 [{}_A \la x_1 \cdot (1\otimes f)b_1 qb_2^* (1\otimes f), \, x_2 \ra \otimes f]a_2^* p .
\end{align*}
On the other hand,
\begin{align*}
& {}_{pM_k (A)p} \la pa_1 \cdot (x_1 \otimes f)\cdot b_1 q, \, pa_2 \cdot (x_1 \otimes f)\cdot b_2 q \ra \\
& =pa_1 (1\otimes f) {}_{M_k (A)} \la (x_1 \otimes f)\cdot b_1 q , \, (x_2 \otimes f)\cdot b_2 q \ra (1\otimes f)a_2^* p \\
& =pa_1 [ {}_A \la x_1 \cdot (1\otimes f)b_1 qb_2^* (1\otimes f), \, x_2 \ra \otimes f]a_2^* p .
\end{align*}
Hence $\Phi$ preserves the left $pM_k (A)p$-valued inner products.
Also,
\begin{align*}
& \la pa_1 (1\otimes f)\otimes x_1 \otimes (1\otimes f)b_1 q , \, pa_2 (1\otimes f)\otimes x_2 \otimes (1\otimes f)b_2 q \ra_{qM_k (B)q} \\
& =\la x_1 \otimes (1\otimes f)b_1 q , \, \la pa_1 (1\otimes f), \, pa_2 (1\otimes f) \ra_A \cdot x_2 \otimes (1\otimes f)b_2 q \ra_{qM_k (B)q} \\
& =\la x_1 \otimes (1\otimes f)b_1 q, \, (1\otimes f)a_1^* pa_2 (1\otimes f)\cdot x_2 \otimes (1\otimes f)b_2 q \ra_{qM_k (B)q} \\
& =\la (1\otimes f)b_1 q, \, [\la x_1 , \, (1\otimes f)a_1^*pa_2 (1\otimes f)\cdot x_2 \ra_B \otimes f]b_2 q \ra_{qM_k (B)q} \\
& =qb_1^* (1\otimes f) [\la x_1 , \, (1\otimes f)a_1^* pa_2 (1\otimes f)\cdot x_2 \ra_B \otimes f]b_2 q \\
& =qb_1^* [\la x_1 , \,  (1\otimes f)a_1^* pa_2 (1\otimes f)\cdot x_2 \ra_B\otimes f]b_2 q .
\end{align*}
On the other hand,
\begin{align*}
& \la pa_1 \cdot (x_1 \otimes f)\cdot b_1 q, \, pa_2 \cdot (x_2 \otimes f)\cdot b_2 q \ra_{qM_k (B)q} \\
& =qb_1^* (1\otimes f)\la pa_1 \cdot (x_1 \otimes f), \, pa_2 \cdot (x_2 \otimes f) \ra_{M_k (B)}(1\otimes f)b_2 q \\
& =qb_1^* [\la x_1 , \, (1\otimes f)a_1^* pa_2 (1\otimes f)\cdot x_2 \ra_B \otimes f]b_2 q .
\end{align*}
Thus $\Phi$ preserves the right $qM_k (B)q$-valued inner products.
Furthermore, let $\{f_{ij}\}_{i,j=1}^k$ be a system of matrix units of $M_k (\BC)$.
Then since $f=f_{11}$, for any $x\in X$ and $i, j=1,2,\dots,k$,
\begin{align*}
p(1\otimes f_{i1})\otimes x\otimes (1\otimes f_{1j})q & =p(1\otimes f_{i1})(1\otimes f)\otimes x\otimes (1\otimes f)(1\otimes f_{1j})q \\
& \in pM_k (A)(1\otimes f)\otimes_A X\otimes_B (1\otimes f)M_k (B)q .
\end{align*}
Then by the definition of $p\cdot M_k (X)\cdot q$, for $i, j=1,2,\dots,k$,
$$
\Phi(p(1\otimes f_{i1})\otimes x\otimes (1\otimes f_{1j})q) =p(1\otimes f_{i1})\cdot (x\otimes f)\cdot (1\otimes f_{1j})q
=p\cdot (x\otimes f_{ij})\cdot q .
$$
This means that $\Phi$ is surjective. Therefore, we obtain the conclusion.
\end{proof}

\begin{cor}\label{cor:bimodule2}With the above notations,
$$
pM_k (C)(1\otimes f)\otimes _C Y\otimes_D (1\otimes f)M_k (D)q\cong p\cdot M_k (Y)\cdot q
$$
as $pM_k (C)p-qM_k (D)q$-equivalence bimodules. Hence $Y_2 \cong p\cdot M_k (Y)\cdot q$
as $C_2 -D_2$-equivalence bimodules, where we identify $pM_k (C)p$ and $qM_k (D)q$ with
$C_2$ and $D_2$, respectively.
\end{cor}
\begin{proof}
This is immediate by Lemma \ref{lem:bimodule}.
\end{proof}

By the above discussions, we can obtain the $C_1 -D_1$-equivalence bimodule isomorphism $\overline{\Phi_1}$
from $Y_2$ onto $p\cdot M_k (Y)\cdot q$ defined by
$$
\overline{\Phi_1}(c_1 \otimes y\otimes \widetilde{d_1})=[E^C(w_i^* c_1 )\cdot y\cdot E^D (d_1^* z_j )]_{i, j=1}^k
$$
for any $c_1 \in C_1$, $d_1 \in D_1$, $y\in Y$, where we identify $C_1$ and $D_1$ with
$pM_k (C)p$ and $qM_k (D)q$ by the isomorphisms defined above, respectively. Also, we can obtain
the $C-D$-equivalence bimodule isomorphism $\overline{\Phi}$ from $Y_1$ onto $p\cdot M_k (X)\cdot q$ defined by
$$
\overline{\Phi}(c\otimes x\otimes d) =[E^A (u_i^* c)\cdot x\cdot E^B (d^* v_j )]_{i, j=1}^k
$$
for any $c\in C$, $d\in D$, $x\in X$, where we identify $C$ and $D$
with $pM_k (A)p$ and $qM_k (B)q$ by the isomorphisms defined above, respectively.
\par
Let $E^{p\cdot M_k (X)\cdot q}$ be the conditional expectation from $p\cdot M_k (Y)\cdot q$
onto $p\cdot M_k (X)\cdot q$ defined by
$$
E^{p\cdot M_k (X)\cdot q}=(E^X \otimes \id_{M_k (\BC)})|_{p\cdot M_k (Y)\cdot q}
$$
with respect to conditional expectations induced by $E^A \otimes \id_{M_k (\BC)}$
and $E^B \otimes\id_{M_k (\BC)}$.

\begin{lemma}\label{lem:composition}With the above notations, we have
$$
E^{p\cdot M_k (X)\cdot q} \circ \overline{\Phi_1}=\overline{\Phi}\circ E^{Y_1} .
$$
\end{lemma}
\begin{proof}
We can prove this lemma by routine computations. Indeed, for any $c_1 \in C_1$, $d_1 \in D_1$, $y\in Y$,
\begin{align*}
(E^{p\cdot M_k (X)\cdot q}\circ \overline{\Phi_1})(c_1 \otimes y\otimes \widetilde{d_1}) 
& =E^{p\cdot M_k (X)\cdot q}([E^C (w_i^* c_1 )\cdot y\cdot E^D (d_1^* z_j )]_{i, j=1}^k ) \\
& =[E^X (E^C (w_i^* c_1 )\cdot y\cdot E^D (d_1^* z_j )]_{i, j=1}^k .
\end{align*}
Let $c_1 =c_2 e_A c_3$, \, $c_2 , c_3 \in C$ and $d_1 =d_2 e_B d_3 $,  \, $d_2 , d_3 \in D$.
We note that for any $i, j=1,2,\dots,k$,
$$
w_i =u_i e_A \Ind_W (E^A )^{\frac{1}{2}} , \quad z_j =v_j e_B \Ind_W (E^B )^{\frac{1}{2}}.
$$
Hence
\begin{align*}
& [E^X (E^C (w_i^* c_1 )\cdot y\cdot E^D (d_1^* z_j ))]_{i, j=1}^k \\
& =[E^X (E^C (\Ind_W (E^A )^{\frac{1}{2}}e_A u_i^* c_2 e_A c_3 )\cdot y\cdot
E^D (d_3^* e_B d_2^* v_j e_B \Ind_W (E^B )^{\frac{1}{2}}))]_{ij}^k \\
& =[E^X (\Ind_W (E^A )^{-\frac{1}{2}}E^A (u_i^* c_2 )c_3 \cdot y\cdot
d_3^* \,E^B (d_2^* v_j )\Ind_W (E^B )^{-\frac{1}{2}})]_{ij=1}^k \\
& =[\Ind_W (E^A )^{-\frac{1}{2}}E^A (u_i^* c_2 )\cdot E^X (c_3 \cdot y\cdot d_3^* )
\cdot E^B (d_2^* v_j )\Ind_W (E^B )^{-\frac{1}{2}}]_{ij=1}^k \\
& =\Ind_W (E^A )^{-1}[E^A (u_i^* c_2 )\cdot E^X (c_3 \cdot y\cdot d_3^* )\cdot E^B (d_2^* v_j )]_{ij=1}^k
\end{align*}
by Lemma \ref{lem:index}. On the other hand,
\begin{align*}
E^{Y_1} (c_1 \otimes y\otimes\widetilde{d_1} ) & =\Ind_W (E^A )^{-1}c_1 \cdot y\cdot d_1^* =\Ind_W (E^A )^{-1}c_1 \cdot\phi(y)\cdot d_1^* \\
& =\sum_{i, j}\Ind_W (E^A )^{-1}c_1 \cdot u_i \otimes E^X (u_i^* \cdot y\cdot v_j )\otimes \widetilde{v_j}\cdot d_1^* 
\end{align*}
Since $c_1 =c_2 e_A c_3 $ and $d_1 =d_2 e_B d_3 $,
$$
E^{Y_1} (c_1 \otimes y\otimes\widetilde{d_1})=\sum_{i, j}\Ind_W (E^A )^{-1}c_2 E^A (c_3 u_i )\otimes E^X (u_i^* \cdot y\cdot v_j )
\otimes [d_2 E^B (d_3 v_j )]^{\widetilde{}} .
$$
Hence
\begin{align*}
& (\overline{\Phi}\circ E^{Y_1})(c_1 \otimes y\otimes \widetilde{d_1}) \\
& =\sum_{i, j}\Ind_W (E^A )^{-1}[(E^A (u_l^* c_2 E^A (c_3 u_i ))\cdot E^X (u_i^* \cdot y\cdot v_j )\cdot
E^B (E^B (v_j^* d_3^* )d_2^* v_m )]_{l, m=1}^k \\
& =\sum_{i, j}\Ind_W (E^A )^{-1}[E^A (u_l^* c_2 )E^A (c_3 u_i )\cdot E^X (u_i^* \cdot y\cdot v_j )
\cdot E^B (v_j^* d_3^* )E^B (d_2^* v_m )]_{l, m=1}^k \\
& =\sum_{i, j}\Ind_W (E^A )^{-1}[E^A (u_l^* c_2 )\cdot E^X (E^A (c_3 u_i )u_i^* \cdot y\cdot v_j E^B (v_j^* d_3^* ))\cdot
E^B (d_2^* v_m )]_{l, m=1}^k \\
& =\Ind_W (E^A )^{-1}[E^A (u_l^* c_2 )\cdot E^X (c_3 \cdot y \cdot d_3^* )\cdot E^B (d_2^* v_m )]_{l, m=1}^k .
\end{align*}
Therefore, we obtain the conclusion.
\end{proof}

\begin{thm}\label{thm:duality}Let $A\subset C$ and $B\subset D$ be
unital inclusions of unital $C^*$-algebras, which are strongly Morita equivalent with
respect to a $C-D$-equivalence bimodule $Y$ and its closed subspace $X$.
Let $E^A$ and $E^B$ be conditional expectations of Watatani index-finite type from
$C$ and $D$ onto $A$ and $B$, respectively and let $E^X$ be a conditional expectation
from $Y$ onto $X$ with respect to $E^A$ and $E^B$.
Let $C_1$, $D_1$ and $Y_1$ be the $C^*$-
basic constructions and the upward basic construction for $E^A$, $E^B$ and $E^X$, respectively.
Also, let $E^C$, $E^D$ and $E^Y$
be the dual conditional expectations from $C_1$, $D_1$ and $Y_1$ onto $C$, $D$ and $Y$,
respectively. Furthermore, in the same way as above, we define the $C^*$-basic constructions and the upward basic
constructions $C_2$, $D_2$ and $Y_2$ for $E^C$, $E^D$ and $E^Y$, respectively and
we define the second dual conditional expectations $E^{C_1}$, $E^{D_1}$ and $E^{Y_1}$,
respectively. Then there are a positive integer $k$ and full projections $p\in M_k (A)$ and $q\in M_k (B)$
with
\begin{align*}
& pM_k (A)p\cong C_1 , \quad qM_k (B)q\cong D_1 , \\
& pM_k (C)p\cong C_2 , \quad qM_k (D)q\cong D_2
\end{align*}
such that there are a $C_1 -D_1$-eqivalence bimodule isomorphism $\overline{\Phi}$ of
$Y_1$ onto $p\cdot M_k (X)\cdot q$ and a $C_2 -D_2$-equivalence bimodule
isomorphism $\overline{\Phi_1}$ of $Y_2$ onto $p\cdot M_k (Y)\cdot q$ satifying
that
$$
E^{p\cdot M_k (X)\cdot q}\circ\overline{\Phi_1}=\overline{\Phi}\circ E^{Y_1}
$$
where $E^{p\cdot M_k (X)\cdot q}$ is the conditional expectation from $p\cdot M_k(Y)\cdot q$
onto $p\cdot M_k (X)\cdot q$ defined by
$$
E^{p\cdot M_k (X)\cdot q}=(E^X \otimes\id_{M_k (\BC)})|_{p\cdot M_k (Y)\cdot q} \,.
$$
\end{thm}
\begin{proof}This is immediate by Lemmas \ref{lem:bimodule}, \ref{lem:composition} and Corollary \ref{cor:bimodule2}.
\end{proof}

\section{The downward basic construction}\label{sec:downward}
Let $A\subset C$ and $B\subset D$ be unital inclusions of unital $C^*$-algebras which are strongly Morita equivalent
with respect to a $C-D$-equivalence bimodule $Y$ and its closed subspace $X$. Let $E^A$ and $E^B$ be
conditional expectations of Watatani index-finite type from $C$ and $D$ onto $A$ and $B$, respectively.
Let $E^X$ be a conditional expectation from $Y$ onto $X$ with respect to $E^A$ and $E^B$.
We suppose that $\Ind_W (E^A )\in A$. Then by Lemma \ref {lem:index2}, $\Ind_W (E^B )\in B$. Also,
we suppose that there are full projections $p$ and $q$ in $C$ and $D$ satisfying that
$$
E^A (p)=\Ind_W (E^A )^{-1} , \quad E^B (q)=\Ind_W (E^B )^{-1} ,
$$
respectively. Then by \cite [Proposition 2.6]{OKT:rohlininclusion},
we obtain the following: Let $P=\{p\}' \cap A$ and let $E^P$
be the conditional expectation from $A$ onto $P$ defined by
$$
E^P (a)=\Ind_W (E^A )E^A (pap)
$$
for any $a\in A$. Similarly, let $Q=\{q\}' \cap B$ and let $E^Q$ be the conditional
expectation from $B$ onto $Q$ defined by
$$
E^Q (b)=\Ind_W (E^B )E^B (qbq)
$$
for any $b\in B$. Then
$\Ind_W (E^P )=\Ind_W (E^A )\in P\cap C'$ and $\Ind_W (E^Q )=\Ind_W (E^B )\in Q\cap D'$.
Furthermore, we can see that
\begin{align*}
ApA & =C, \quad BqB=D , \\
pap & =E^P (a), \quad qbq=E^Q (b)
\end{align*}
for any $a\in A$ and $b\in B$. Also, the unital inclusions $A\subset C$ and $B\subset D$
can be regarded as the $C^*$-basic constructions of the unital inclusions $P\subset A$ and $Q\subset B$, respectively.
In this section, we shall show that the unital inclusions $P\subset A$ and $Q\subset B$
are strongly Morita equivalent and that there is a conditional expectation from $X$ onto
its closed subspace with respect to $E^P$ and $E^Q$.
\par
Let $Z=\{x\in X \, | \, p\cdot x=x\cdot q \}$. Then $Z$ is a closed subspace of $X$.

\begin{lemma}\label{lem:subspace}With the above notations, $Z$ is a Hilbert
$P-Q$-bimodule in the sense of Brown, Mingo and Shen \cite{BMS:quasi}.
\end{lemma}
\begin{proof}
This lemma can be proved by routine computations. Indeed,
for any $a\in P$, $x\in Z$,
$$
p\cdot (a\cdot x)=pa\cdot x=a\cdot (p\cdot x)=a\cdot (x\cdot q)=(a\cdot x)\cdot q .
$$
Hence $a\cdot x\in Z$ for any $a\in P$, $x\in Z$. Similarly for any $b\in Q$, $x\in Z$, $x\cdot b\in Z$.
For any $x, y\in Z$,
$$
p\cdot {}_A \la x, y \ra ={}_C \la p\cdot x , y \ra ={}_C \la x\cdot q , y \ra ={}_C \la x, p\cdot y \ra 
={}_A \la x, y \ra \cdot p .
$$
Hence ${}_A \la x, y \ra \in P$ for any $x, y\in Z$. Similarly for any $x, y\in Z$, $\la x, y \ra_A\in Q$.
Since $Z$ is a closed subspace of the $A-B$-equivalence bimodule $X$, $Z$ is a Hilbert
$P-Q$-bimodule in the sense of Brown, Mingo and Shen \cite {BMS:quasi}.
\end{proof}

Let $E^Z$ be the linear map from $X$ to $Z$ defined by
$$
E^Z (x)=\Ind_W (E^A )\cdot E^X (p\cdot x\cdot q)
$$
for any $x\in X$. We note that
$$
E^Z (x)=E^X (p\cdot x\cdot q)\cdot\Ind_W (E^B )
$$
for any $x\in X$ by Lemma \ref{lem:index}.

\begin{lemma}\label{lem:satisfies}With the above notations, $E^Z$ satisfies Conditions (1)-(6)
in Definition \ref{def:expectation}.
\end{lemma}
\begin{proof}
For any $a\in A$, $z\in Z$,
\begin{align*}
E^Z (a\cdot z) & =\Ind_W (E^A )\cdot E^X (p\cdot (a\cdot z)\cdot q)
=\Ind_W (E^A )\cdot E^X (pa\cdot z\cdot q) \\
& =\Ind_W (E^A )\cdot E^X (pap\cdot z)=\Ind_W (E^A )E^A (pap)\cdot z =E^P (a)\cdot z .
\end{align*}
Hence $E^Z$ satisfies Condition (1) in Definition \ref{def:expectation}.
Similarly $E^Z$ satisfies Condition (4) in Definition \ref{def:expectation}.
For any $b\in Q$, $x\in X$,
\begin{align*}
E^Z (x\cdot b) & =\Ind_W (E^A )\cdot E^X (p\cdot (x\cdot b)\cdot q)=\Ind_W (E^A )\cdot E^X (p\cdot x\cdot qb) \\
& =\Ind_W (E^A )\cdot E^X (p\cdot x\cdot q)\cdot b=E^Z (x)\cdot b .
\end{align*}
Hence $E^Z$ satisfies Condition (5) in Definition \ref{def:expectation}. Similarly $E^Z$
satisfies Condition (2) in Definition \ref{def:expectation}. For any $x\in X$, $z\in Z$,
\begin{align*}
{}_P \la E^Z (x), \, z \ra & = {}_A \la \Ind_W (E^A )\cdot E^X (p\cdot x\cdot q), \, z \ra
=\Ind_W (E^A ){}_A \la E^X (p\cdot x\cdot q), \, z \ra \\
& =\Ind_W (E^A )E^A ({}_A \la p\cdot x\cdot q, \, z \ra )
=\Ind_W (E^A )E^A (p{}_A \, \la x, \, z\cdot q \ra ) \\
& =\Ind_W (E^A )E^A (p{}_A \, \la x, \, p\cdot z \ra )
=\Ind_W (E^A )E^A (p{}_A \, \la x, z \ra p) \\
& =E^P ({}_A \la x, z \ra ) .
\end{align*}
Hence $E^Z$ satisfies Condition (3) in Definition \ref{def:expectation}.
Also, in the same way as above, by Lemma \ref{lem:index},
we can see that $E^Z$ satisfies Condition (6) in Definition \ref{def:expectation}.
\end{proof}

\begin{lemma}\label{lem:innerproduct}With the above notations,
${}_A \la X, Z \ra=A$, $\la X, Z \ra_B =B$.
\end{lemma}
\begin{proof}Since $E^Z$ is surjective by Lemma \ref {lem:satisfies},
\begin{align*}
{}_A \la X, Z \ra & ={}_A \la X, \, E^Z (X) \ra ={}_A \la X, \, \Ind_W (E^A )\cdot E^X (p\cdot X\cdot q) \ra \\
& ={}_A \la X, \, E^X (p\cdot X \cdot q) \ra \Ind_W (E^A )=E^A ({}_C \la X, \, p\cdot X\cdot q \ra )\Ind_W (E^A ) \\
& =E^A ({}_C \la X, \, X\cdot q \ra p)\Ind_W (E^A ) .
\end{align*}
Since $X\cdot B=X$ by \cite [Proposition 1.7]{BMS:quasi} and $BqB=D$,
\begin{align*}
{}_A \la X, \, Z \ra & =E^A ({}_C \la X\cdot B, \, X\cdot Bq \ra p)\Ind_W (E^A )
=E^A ({}_C \la X, \, X\cdot BqB \ra p )\Ind_W (E^A ) \\
& =E^A ({}_C \la X, \, X\cdot D \ra p)\Ind_W (E^A ) .
\end{align*}
Since $B\subset D$, $ X=X\cdot B \subset X\cdot D$ by \cite [Proposition 1.7]{BMS:quasi}.
Hence
\begin{align*}
{}_A \la X, \, Z \ra & \supset E^A ({}_C \la X, \, X \ra p)\Ind_W (E^A )
=E^A ({}_A \la X, \, X \ra p)\Ind_W (E^A ) \\
& =E^A (Ap)\Ind_W (E^A )=A .
\end{align*}
Since ${}_A \la X, \, Z \ra \subset A$, we obtain that ${}_A \la X, \, Z \ra =A$.
Similarly we obtain that $\la X, \, Z \ra_B =B$. Therefore we obtain the conclusion.
\end{proof}

\begin{cor}\label{cor:expectation3}With the above notations,
$Z$ is a $P-Q$-equivalence bimodule and $E^Z$ is a conditional
expectation from $X$ onto $Z$ with respect to $E^P$ and $E^Q$.
\end{cor}
\begin{proof}First, we show that $Z$ is a $P-Q$-equivalence bimodule.
By Lemma \ref{lem:subspace}, we have only to show that $Z$ is full with the both
sided inner products. Since $E^Z$ is surjective by Lemma \ref{lem:satisfies},
\begin{align*}
{}_P \la Z, Z \ra & ={}_P \la E^Z (X), E^Z (X) \ra
=E^P ({}_A \la X, E^Z (X)\ra )
=E^P ({}_A \la X, Z \ra) \\
& =E^P (A)=P
\end{align*}
by Lemma \ref{lem:innerproduct}.
Similarly $\la Z, Z \ra_Q =Q$. Thus, $Z$ is a $P-Q$-equivalence bimodule.
Hence $E^Z$ is a conditional expectation from $X$ onto $Z$ with respect to
$E^P$ and $E^Q$.
\end{proof}

\begin{prop}\label{prop:equivalence}With the above notations, unital inclusions $P\subset A$
and $Q\subset B$ are strongly Morita equivalent with respect to the $P-Q$-
equivalence bimodule $X$ and its closed subspace $Z$ and there is a conditional expectation from
$X$ onto $Z$ with respect to $E^P$ and $E^Q$.
\end{prop}
\begin{proof}This is immediate by Lemmas \ref{lem:subspace}, \ref{lem:satisfies} and Corollary \ref{cor:expectation3}.
\end{proof}

\begin{Def}\label{def:predual}In the above situation, $Z$ is called the
\it
downward basic construction
\rm
of $X$ for $E^X$. Also, $E^Z$ is called the
\it
pre-dual
\rm
conditional expectation of $E^X$.
\end{Def}

\section{Relation between the upward basic construction and the downward basic construction}\label{sec:between}
Let $A\subset C$ and $B\subset D$ be unital inclusions of unital $C^*$-algebras,
which are strongly Morita equivalent with respect to a $C-D$-equivalence bimodule $Y$ and its closed subspace
$X$. Let $E^A$ and $E^B$ be conditional expectations of Watatani index-finite type from $C$ and $D$ onto
$A$ and $B$, respectively. Let $E^X$ be a conditional expectation from $Y$ onto $X$
with respect to $E^A$ and $E^B$. We suppose
that $\Ind_W (E^A )\in A$ and $\Ind_W (E^B )\in B$.
Let $e_A$ and $e_B$ be the Jones' projections for $E^A$ and $E^B$, respectively.
Then by \cite [Lemma 2.1.1]{Watatani:index},
$$
A=\{a\in C \, | \, e_A a=ae_A \}, \quad B=\{b\in D \, | \, e_B b=be_B \},
$$
respectively. Let $C_1$ and $D_1$ be the $C^*$-basic constructions for $E^A$ and $E^B$,
respectively and let $E^C$ and $E^D$ be the dual conditional expectations from $C_1$
and $D_1$ onto $C$ and $D$, respectively. Then $e_A$ and $e_B$ are full projections in
$C_1$ and $D_1$, respectively by \cite [Lemma 2.1.6]{Watatani:index} and
$$
\Ind_W (E^C )=\Ind_W (E^A )\in A, \quad \Ind_W (E^D )=\Ind_W (E^B )\in B ,
$$
respectively.
Furthermore,
\begin{align*}
E^A (x) & =\Ind_W (E^C )E^C (e_A xe_A ) \quad \text{for any $x\in C$} , \\
E^B (x) & =\Ind_W (E^D )E^D (e_B xe_B ) \quad \text{for any $x\in D$} , 
\end{align*}
respectively. Let $Y_1$ be the upward basic construction for $E^X$ and $E^Y$
the dual conditional expectation of $E^X$ from $Y_1$ onto $Y$. We recall that
$Y$ can be regarded as a closed subspace of $Y_1$ by the linear map $\phi$ from
$Y$ to $Y_1$ defined by
$$
\phi(y)=\sum_{i, j}u_i \otimes E^X (u_i^* \cdot y\cdot v_j )\otimes \widetilde{v_j} ,
$$
for any $y\in Y$, where $\{(u_i , u_i^* )\}$ and $\{(v_j , v_j^* )\}$ are
quasi-bases for $E^A$ and $E^B$, respectively and
$$
Y_1 =C\otimes_A X \otimes _B \widetilde{D} .
$$
Let
$$
Z=\{y\in Y \, | \, e_A \cdot\phi(y)=\phi(y)\cdot e_B \} .
$$
By the discussions in Section \ref{sec:downward}, $Z$ is a closed subspace of $Y$ and
$Z$ is an $A-B$-equivalence bimodule.

\begin{lemma}\label{lem:coincide}With the above notations, $Z=X$
\end{lemma}
\begin{proof}
For any $x\in X$,
\begin{align*}
e_A \cdot \phi(x) & =\sum_{i, j}e_A \cdot u_i\otimes E^X (u_i^* \cdot x\cdot v_j )\otimes \widetilde{v_j} \\
& =\sum_{i, j}1\otimes E^X (E^A (u_i )u_i^* \cdot x\cdot v_j )\otimes \widetilde{v_j} \\
& =\sum_j 1\otimes E^X (x\cdot v_j )\otimes\widetilde{v_j} =\sum_j 1\otimes x\cdot E^B (v_j )\otimes \widetilde{v_j} \\
& =\sum_j 1\otimes x\otimes [v_j E^B (v_j^* )]^{\widetilde{}} =1\otimes x \otimes \widetilde{1} .
\end{align*}
Similarly, $\phi(x)\cdot e_B =1\otimes x\otimes \widetilde{1}$. Hence $x\in Z$.
Thus $X\subset Z$. Also, let $y\in Z$. Since $e_A \cdot \phi(y)=\phi(y)\cdot e_B$,
$$
e_A \cdot \phi(y)=e_A^2 \cdot \phi(y)=e_A \cdot \phi(y)\cdot e_B .
$$
Also, since
$$
e_A \cdot\phi(y)=\sum_j 1\otimes E^X (y\cdot v_j )\otimes\widetilde{v_j} \quad\text{and}\quad
e_A \cdot\phi(y)\cdot e_B =1\otimes E^X (y)\otimes\widetilde{1} ,
$$
we see that
$$
\sum_j 1\otimes E^X (y\cdot v_j )\otimes \widetilde{v_j}=1\otimes E^X (y)\otimes\widetilde{1} .
$$
Using the conditional expectation $E^Y$,
$$
\Ind_W (E^A )^{-1}\cdot E^X (y) =\sum_j \Ind_W (E^A )^{-1}\cdot E^X (y\cdot v_j )\cdot v_j^* =\Ind_W (E^A )^{-1}\cdot y
$$
by Lemma \ref {lem:basis}. Thus $E^X (y)=y$, that is, $y\in X$. Therefore, we obtain the conclusion.
\end{proof}

By Lemmas \ref{lem:commute} and \ref{lem:coincide}, we obtain the following:
\begin{prop}\label{prop:expectation4}With the above notations, $X$ can be regarded as the downward
basic construction for $E^Y$ and $E^X$ can be regarded as the pre-dual conditional
expectation of $E^Y$.
\end{prop}

Next, let $p$ and $q$ be full projections in $C$ and $D$ satisfying that
$$
E^A(p)=\Ind_W (E^A )^{-1} , \quad E^B (q)=\Ind_W (E^B )^{-1} ,
$$
respectively.
Let $P, Q, E^P, E^Q$ and $Z, E^Z$ be as in Section \ref {sec:downward}.
We shall show that $Y$ is the upward basic construction for $E^Z$ and that $E^X$ is the
dual conditional expectation of $E^Z$. By Section \ref{sec:downward},
we can see that
$$
\Ind_W (E^P )=\Ind_W (E^A )\in P\cap C' , \quad \Ind_W (E^Q )=\Ind_W (E^B )\in Q\cap D' .
$$
Also, we can see that
$$
E^Z (x)=\Ind_W (E^A )\cdot E^X (p\cdot x\cdot q) .
$$
Furthermore, we can regard $C$ and $D$
as the $C^*$-basic constructions for $E^P$ and $E^Q$, respectively by \cite [Proposition 2.6]{OKT:rohlininclusion}.
We can also regard $p$ and $q$ as the Jones projections in
$C$ and $D$, respectively. Hence by Proposition \ref{prop:universal}, we obtain the following proposition:

\begin{prop}\label{prop:updown}With the above notations, $Y$ can be regarded as the upward basic construction
for $E^Z$ and $E^X$ can be regarded as the dual conditional expectation of $E^Z$.
\end{prop}

\section{The strong Morita equivalence and the paragroups}\label{sec:paragroup}
In this section, we show that the strong Morita equivalence for unital
inclusions of unital $C^*$-algebras preserves their paragroups.
We begin this section with the following easy lemmas:

\begin{lemma}\label{lem:easy1}Let $A\subset C$ and $B\subset D$ be unital inclusions of
unital $C^*$-algebras, which are strongly Morita equivalent with respect to a $C-D$-equivalence
bimodule $Y$ and its closed subspace $X$. Then $C\cdot X=X\cdot D =Y$.
\end{lemma}
\begin{proof}Since $X$ is an $A-B$-equivalence bimodule and $A\subset C$ is a unital
inclusion, there are elements $x_1 , x_2 , \dots x_n \in X$ such that $\sum_{i=1}^n \la x_i , x_i \ra_B =1_D $.
Then for any $y\in Y$,
$$
y=y\cdot 1_D =\sum_{i=1}^n y\cdot \la x_i , x_i \ra_B =\sum_{i=1}^n {}_C \la y , x_i \ra \cdot x_i .
$$
Hence we can see that $C\cdot X =Y$. Similarly we obtain that $X\cdot D =Y$.
\end{proof}

Let $A\subset C$ and $B\subset D$ be as above. Let $C\subset C_1$ and $D\subset D_1$
be unital inclusion of unital $C^*$-algebras, which are strongly Morita equivalent with
respect to a $C_1 -D_1$-equivalence bimodule $Y_1$ and its closed subspace $Y$.
We note that $X\subset Y\subset Y_1$.

\begin{lemma}\label{lem:easy2}With the above notations, the inclusion $A\subset C_1$ and
$B\subset D_1$ are strongly Morita equivalent with respect to the $C_1 -D_1$-equivalence
bimodule $Y_1$ and its closed subspace $X$.
\end{lemma}
\begin{proof}It suffices to show that
$$
{}_{C_1} \la Y_1 , \, X \ra =C_1 \quad \la Y_1 , \, X \ra_{D_1}=D_1 .
$$
Indeed, by \cite [Proposition 1.7]{BMS:quasi} and Lemma \ref{lem:easy1},
\begin{align*}
{}_{C_1} \la Y_1 ,\, X \ra & ={}_{C_1} \la Y_1 \cdot D_1 \, , X \ra = {}_{C_1} \la Y_1 , \, X\cdot D_1 \ra
={}_{C_1} \la Y_1 , \, X\cdot DD_1 \ra \\
& ={}_{C_1} \la Y_1 , \, Y\cdot D_1 \ra ={}_{C_1} \la Y_1 , \, Y_1 \ra =C_1 .
\end{align*}
Similarly, we can prove that $\la Y_1 \, , X \ra_{D_1} =D_1$.
\end{proof}

Let $A\subset C$ and $B\subset D$ be unital inclusions of unital $C^*$-algebras,
which are strongly Morita equivalent
with respect to a $C-D$-equivalence bimodule $Y$ and its closed subspace $X$.
Then by Lemmas \ref{lem:bijection}, \ref{lem:condition} and Corollary \ref{cor:restriction},
we may assume that
$$
B=pM_n (A)p, \, D=pM_n (C)p, \, Y=(1\otimes f)M_n (C)p, \, X=(1\otimes f)M_n (A)p,
$$
where $p$ is a full projection in $M_n (A)$
and $n$ is a positive integer. We regard $X$ and $Y$ as an
$A-pM_n (A)p$-equivalence bimodule and a $C-pM_n (C)p$-equivalence bimodule in
the usual way.

\begin{lemma}\label{lem:commutant}With the above notations, we suppose that
unital inclusions of unital $C^*$-algebras $A\subset C$ and $B\subset D$ are
strongly Morita equivalent. Then the relative commutants $A' \cap C$ and $B' \cap D$ are
isomorphic.
\end{lemma}
\begin{proof}By the above discussions, we have only to show that
$$
A' \cap C \cong (pM_n (A)p)' \cap pM_n (C)p ,
$$
where $p$ is a projection in $M_n (A)$ satisfying the above. By routine computations,
we can see that
$$
M_n (A)' \cap M_n (C)=\{c\otimes I_n \, | \, c\in A' \cap C \} .
$$
Hence we can see that $A' \cap C\cong M_n (A)' \cap M_n (C)$. Next, we claim that
$M_n (A)' \cap M_n (C)\cong (M_n (A)\cap M_n (C))p$. Indeed, let $\pi$ be the map
from $M_n (A)' \cap M_n (C)$ onto $(M_n (A)' \cap M_n (C))p$ defined by $\pi(x)=px$
for any $x\in M_n (A)' \cap M_n (C)$. Since $p$ is a projection in $M_n (A)$, $\pi$ is
a homomorphism of $M_n (A)' \cap M_n (C)$ onto $(M_n (A)' \cap M_n (C))p$.
We suppose that $xp=0$ for an element $x\in M_n (A)' \cap M_n (C)$.
Since $p$ is full in $M_n(A)$, there are elements $z_1 , \dots, z_m \in M_n (A)$ such that
$$
\sum_{i=1}^m z_i pz_i^* =1_{M_n (A)} .
$$
Then
$$
0=\sum_{i=1}^m z_i xpz_i^* =\sum_{i=1}^m xz_i pz_i^* =x .
$$
Hence $\pi$ is injective. Thus $\pi$ is an isomorphism of $M_n (A)' \cap M_n (C)$
onto $(M_n (A)' \cap M_n (C))p$. Finally we show that
$$
(pM_n (A)p)' \cap pM_n (C)p=(M_n (A)' \cap M_n (C))p .
$$
Indeed, by easy computations, we can see that
$$
pM_n (A)p)' \cap pM_n (C)p\supset (M_n (A)' \cap M_n (C))p .
$$
We prove the inverse inclusion. Let $y\in (pM_n (A)p)' \cap pM_n (C)p$. Let
$w=\sum_{i=1}^m z_i y z_i^*$. Then for any $x\in M_n (A)$,
\begin{align*}
wx & =\sum_{i, j=1}^m z_i yz_i^* xz_j pz_j^* =\sum_{i, j=1}^m z_i ypz_i^* xz_j p z_j^*
=\sum_{i, j}^m z_i pz_i^* xz_j pyz_j^* \\
& =\sum_{j=1}^m xz_j pyz_j^* =\sum_{j=1}^m xz_j yz_j^* =xw .
\end{align*}
Hence $w\in M_n(A)' \cap M_n(C)$. On the other hand,
$$
wp= pw=\sum_{i=1}^m pz_i yz_i^* =\sum_{i=1}^m pz_i pyz_i^*
=\sum_{i=1}^m ypz_i pz_i^* =yp=y .
$$
Thus $y\in (M_n (A)' \cap M_n (C))p$. Hence
$$
(pM_n (A)p)' \cap pM_n (C)p=(M_n (A)' \cap M_n (C))p .
$$
Therefore, we obtain the conclusion.
\end{proof}

Let $A\subset C$ and $B\subset D$ be as above. We suppose that there is a conditional expectation
$E^A$ of Watatani index-finite type from $C$ onto $A$. Then by Section \ref{sec:definition},
there are a conditional expectation of Watatani index-finite type from $D$ onto $B$ and
a conditional expectation $E^X$ from $Y$ onto $X$ with respect to $E^A$ and $E^B$.
For any $n\in \BN$, let $C_n$ and $D_n$ be the $n$-th $C^*$-basic constructions
for conditional expectations $E^A$ and $E^B$, respectively. Then by Corollary \ref{cor:dual},
the inclusions $C_{n-1}\subset C_n$ and $D_{n-1}\subset D_n$ are strongly Morita equivalent
for any $n\in \BN$, where $C_0 =C$ and $D_0 =D$. Thus, by Lemma \ref{lem:easy2},
$A\subset C_n$ and $B\subset D_n$ are strongly Morita equivalent for any $n\in \BN$.

\begin{thm}\label{thm:paragroup}Let $A\subset C$ and $B\subset D$ be unital inclusions of
unital $C^*$-algebras, which are strongly Morita equivalent. We suppose that
there is a conditional expectation
of Watatani index-finite type from $C$ onto $A$.
Then the paragroups of $A\subset C$ and $B\subset D$
are isomorphic.
\end{thm}
\begin{proof}This is immediate by the above discussions and Lemma \ref{lem:commutant}.
\end{proof}







\begin{thebibliography}{99}

\bibitem{Blackadar:K-Theory}B. Blackadar, {\it K-theory for operator algebras},
M. S. R. I. Publications 5, 2nd Edition,
Cambridge Univ. Press, Cambridge, 1998.

\bibitem{BCM:crossed}R. J. Blattner, M. Cohen and S. Montgomery,
{\it Crossed products and inner actions of Hopf algebras},
Trans. Amer. Math. Soc.,
{\bf 298}
(1986),
671--711.

\bibitem{Brown:hereditary}L. G. Brown,
{\it Stable isomorphism of hereditary subalgebra of $C^*$-algebras},
Pacific J. Math.,
{\bf 71}
(1977),
335--348.

\bibitem{BGR:linking}L. G. Brown, P. Green and M. A. Rieffel,
{\it Stable isomorphism and strong Morita equivalence of $C^*$-algebras},
Pacific J. Math.,
{\bf 71}
(1977),
349--363.

\bibitem{BMS:quasi}L. G. Brown, J. Mingo and N-T. Shen,
{\it Quasi-multipliers and embeddings of Hilbert $C^*$-bimodules},
Can. J. Math.,
{\bf 46}
(1994),
1150--1174.

\bibitem{Combes:morita}F. Combes,
{\it Crossed products and Morita equivalence},
Proc. London Math. Soc.,
{\bf 49}
(1984),
289--306.

\bibitem{CKRW:equivalence}R. E. Curto, P. S. Muhly and D. P. Williams,
{\it Cross products of strong Morita equivalent $C^*$-algebras},
Proc. Amer. Math. Soc.,
{\bf 90}
(1984),
528--530.

\bibitem{ER:multiplier}S. Echterhoff and I. Raeburn,
{\it Multipliers of imprimitivity bimodules and Morita equivalence of
crossed products},
Math. Scand.,
{\bf 76}
(1995), 289--309.

\bibitem{JW:covariant}S. Jansen and S. Waldmann,
{\it The H-covariant strong Picard groupoid},
J. Pure  Appl. Algebra, 
{\bf 205}
(2006), 542--598.

\bibitem{JT:KK}K. K. Jensen and K. Thomsen,
{\it Elements of KK-theory},
Birkh$\ddot a$user,
1991.

\bibitem{KW1:bimodule}T. Kajiwara and Y. Watatani,
{\it Jones index theory by Hilbert $C^*$-bimodules and K-Thorey},
Trans. Amer. Math. Soc.,
{\bf 352}
(2000), 3429--3472.

\bibitem{KW2:discrete}T. Kajiwara and Y. Watatani,
{\it Crossed products of Hilbert $C^*$-bimodules by countable discrete groups},
Proc. Amer. Math. Soc.,
{\bf 126}
(1998), 841--851.

\bibitem{Kodaka:equivariance}K. Kodaka,
{\it Equivariant Picard groups of $C^*$-algebras with finite dimensional $C^*$-Hopf algebra coactions},
preprint, arXiv:1512.07724, Rocky Mountain J. Math., to appear.

\bibitem{KT1:inclusion}K. Kodaka and T. Teruya,
{\it Inclusions of unital $C^*$-algebras of index-finite type with depth 2 induced by saturated
actions of finite dimensional $C^*$-Hopf algebras},
Math. Scand.,
{\bf 104}
(2009),
221--248.

\bibitem{KT2:coaction}K. Kodaka and T. Teruya,
{\it The Rohlin property for coactions of finite dimensional $C^*$-Hopf algebras on
unital $C^*$-algebras},
J. Operator Theory,
{\bf 74}
(2015),
329--369.

\bibitem{KT3:equivalence}K. Kodaka and T. Teruya,
{\it The strong Morita equivalence for coactions of a finite dimensional $C^*$-Hopf algebra on unital $C^*$-algebras},
Studia Math.,
{\bf 228}
(2015),
259--294.

\bibitem{Lance:toolkit}E. C. Lance, {\it Hilbert $C^*$-modules},
A toolkit for operatros algebraists, London Math. Soc., Lecture Note Series,
{\bf 210}, 
Cambridge Univ. Press, Cambridge, 1995.

\bibitem{MT:Kac}T. Masuda and R. Tomatsu,
{\it Classification of minimal actions of a compact Kac algebra with the amenable dual},
J. Funct. Anal.,
{\bf 258}
(2010),
1965-2025.

\bibitem{OKT:rohlininclusion}H. Osaka, K. Kodaka and T. Teruya,
{\it The Rohlin property for inclusions of $C^*$-algebras with a finite Wtatani index},
Operator structures and dynamical systems, 177--195, Contemp Math.,
{\bf 503}
Amer Math. Soc., Providence, RI, 2009.

\bibitem{Packer:projective}J. A. Packer,
{\it $C^*$-algebras generated by projective representations of the discrete Heisenberg group},
J. Operator Theory,
{\bf 18}
(1987),
41--66.

\bibitem{RW:continuous}I. Raeburn and D. P. Williams,
{\it Morita equivalence and continuous -trace $C^*$-algebras},
Mathematical Surveys and Monographs, {\bf 60}, Amer. Math. Soc., 1998.

\bibitem{Rieffel:rotation}M. A. Rieffel,
{\it $C^*$-algebras associated with irrational rotations},
Pacific J. Math.,
{\bf 93}
(1981),
415--429.

\bibitem{Sweedler:Hopf}
M. E. Sweedler,
{\it Hopf algebras},
Benjamin, New York, 1969.

\bibitem{SP:saturated}W. Szyma\'nski and C. Peligrad,
{\it Saturated actions of finite dimensional Hopf {\rm *}-algebras on  
$C^*$-algebras},
Math. Scand.,
{\bf 75}
(1994),
217--239.

\bibitem{Tomiyama:projection}J. Tomiyama,
{\it On the projection of norm one in $W^*$-algebras},
Japan Acad.,
{\bf 33}
(1957),
608--612.

\bibitem{Watatani:index}Y. Watatani,
{\it Index for $C^*$-subalgebras},
Mem. Amer. Math. Soc.,
{\bf 424}
(1990).


\end{thebibliography}
\end{document}